\documentclass[opre]{informs3aa}
\usepackage{subfigure}
\usepackage{arydshln}
\usepackage{makecell}

\OneAndAHalfSpacedXI 

\usepackage{endnotes}

\usepackage{longtable}
\usepackage{tabularx}
\usepackage{epstopdf}
\usepackage{booktabs}
\usepackage{etex}
\usepackage{multirow}
\usepackage{array}
\usepackage{bbm}

\usepackage[title]{appendix}

\newcommand{\vct}{\boldsymbol }
\newcommand{\mat}{\mathbf }

\newcommand{\tlam}{\tilde{\lambda}}
\newcommand{\tB}{\tilde{B}}
\newcommand{\tilb}{\tilde{b}}

\renewcommand{\tilde}{\widetilde}
\renewcommand{\hat}{\widehat}
\renewcommand{\bar}{\overline}

\newcommand{\mS}{\mathcal S}
\newcommand{\mX}{\mathcal X}

\renewcommand{\hat}{\widehat}
\renewcommand{\tilde}{\widetilde}
\renewcommand{\bar}{\overline}

\usepackage{amsbsy, amsfonts, amsgen, amsmath, amsopn, amssymb, amstext,
amsxtra, bezier, color, enumerate, graphicx, latexsym, verbatim,
pictexwd, supertabular, url, dsfont,leftidx, mathrsfs, appendix, setspace}
\usepackage{algorithm}
\usepackage{algorithmicx}
\usepackage{algpseudocode}
\makeatletter
\def\BState{\State\hskip-\ALG@thistlm}
\makeatother

\usepackage{pifont}

\usepackage{natbib}
 \bibpunct[, ]{(}{)}{,}{a}{}{,}%
 \def\bibfont{\small}%
\usepackage[colorlinks=true,bookmarks=false,urlcolor=blue, citecolor=blue,linkcolor=blue,bookmarksopen=false,draft=false]{hyperref}

\newcommand{\ud}{\mathrm d}

\newcommand{\R}{ \mathbb{R} }

\usepackage{parskip} 

\newcommand{\const}{\mathrm{const}}

\newcommand{\mZ}{\mathcal Z}

\TheoremsNumberedThrough     
\ECRepeatTheorems

\EquationsNumberedThrough    

\MANUSCRIPTNO{} 

\definecolor{DSgray}{cmyk}{0,1,0,0}

\def\unblind{1}

\usepackage[color=yellow]{todonotes}
\usepackage{xcolor}

\definecolor{mygray}{RGB}{195,195,195}

\begin{document}




\RUNTITLE{Re-solving heuristics for Assortment Optimization}

\TITLE{A Re-solving Heuristic for Dynamic Assortment Optimization with Knapsack Constraints}

\ifdefined\unblind
\ARTICLEAUTHORS{%
\AUTHOR{Xi Chen \thanks{Author names listed in alphabetical order. Correspondence should be directed to Yining Wang (yining.wang@utdallas.edu).}}
\AFF{Stern School of Business, New York University, New York, NY 10011, USA}
\AUTHOR{Mo Liu}
\AFF{Department of Statistics and Operations Research, University of North Carolina, Chapel Hill, NC 27599, USA}
\AUTHOR{Yining Wang}
\AFF{Naveen Jindal School of Management, University of Texas at Dallas, Richardson, TX 75080, USA}
\AUTHOR{Yuan Zhou}
\AFF{Yau Mathematical Sciences Center \& Department of Mathematical Sciences, Tsinghua University, Beijing, China}
} 
\fi

\ABSTRACT{In this paper, we consider a multi-stage dynamic assortment optimization problem with multi-nomial choice modeling (MNL) under resource knapsack constraints. Given the current resource inventory levels, the retailer makes an assortment decision at each period, and the goal of the retailer is to maximize the total profit from purchases.  
With the exact optimal dynamic assortment solution being computationally intractable, a practical strategy is to adopt the re-solving technique that periodically re-optimizes deterministic linear programs (LP) arising from fluid approximation. However, the fractional structure of MNL makes the fluid approximation in assortment optimization non-linear, which brings new technical challenges. To address this challenge, we propose a new epoch-based re-solving algorithm that effectively transforms the denominator of the objective into the constraint, {so that the re-solving technique is applied to a \emph{linear program} with additional slack variables amenable to practical computations and theoretical analysis.} Theoretically, we prove that the regret (i.e., the gap between the resolving policy and the optimal objective of the fluid approximation) scales logarithmically with the length of time horizon and resource capacities.
}



\date{}

\maketitle


\section{Introduction}

Assortment optimization is a core problem in revenue management, which has a wide range of applications in online retailing and advertising. For a given set of substitutable products, the assortment optimization problem refers to the selection of a subset of products (a.k.a., an assortment) offered to a customer. For example, the brand Fanta has 7 popular flavors\footnote{According to \url{https://www.fanta.com/products}, there are 7 popular flavors, including Orange, Pineapple, Strawberry, Grape, Pina Colada, Peach, Fruit punch, Zero sugar, and Berry.}. For a Fanta vending machine with at most three slots, the assortment optimization aims at selecting at most three flavors to maximize a certain objective, such as purchasing probability or the expected revenue.  One key in assortment optimization is the modeling of customer choice behavior when facing multiple products. {In practice, the multi-nomial logit (MNL) model is arguably the most widely used choice model (see Eq.~\eqref{eq:defn-mnl} for details),
dating back to the foundational works of \cite{mcfadden1972conditional,McFadden1980,mcfadden2000mixed}.}
Under the MNL, the single-stage assortment optimization problem is essentially a combinatorial optimization problem, which has been extensively studied in the literature \citep{Ryzin1999,Mahajan2001,Talluri2004,Gallego2004,Liu2008,rusmevichientong2010dynamic}. 

In this paper, we study the {multi-stage} assortment optimization under resource inventory constraints. The resource inventory constraint, which is essentially a type of knapsack constraint, has been widely studied in network revenue management for multi-product pricing  \citep{Gallego1997}. Similar to the case of network revenue management, exact solutions to multi-period assortment optimization with resource constraints require an exponential number of states in a Markov decision process and is therefore computationally intractable.
A fluid approximation is desirable, which replaces hard resource constraints with ones that only hold in expectation that greatly simplifies the computational burden
of finding an (approximately) optimal solution.
Heuristic methods such as probability sampling or resource tracking could then be proposed to match such fluid approximation solutions.

With the optimal solution to a fluid approximation formulation obtained,
one effective approach to match such approximation in dynamic pricing is called ``re-solving'' (a.k.a., ``re-optimization'') method, which periodically resolves the fluid approximation based on the left inventory level. For various dynamic pricing problems (e.g., single product v.s. network revenue management, different demand models, the presence of reusable products), many re-solving techniques have been proposed and analyzed in the recent literature (see, e.g., \cite{jasin2014reoptimization, chen2015real,chen2019nonparametric,lei2020real,Wang2020resolve,Wang2020constant,lei2018randomized,maglaras2006dynamic,cooper2002asymptotic,reiman2008asymptotically,secomandi2008analysis,vera2021online}). For network revenue management, the state-of-the-art results from \cite{Wang2020resolve, Wang2020constant} establish constant regrets, which improve the classical $\log(T)$ regret bound \citep{jasin2014reoptimization}. The re-solving technique has also been applied to linear contextual bandits with knapsack constraints \citep{Xu21reopt}.

Despite the vast literature on re-solving for network revenue management, there is little research on re-solving for assortment optimization. One key challenge is due to the special structure of the MNL choice model (see Eq.~\eqref{eq:defn-mnl} in Sec.~\ref{sec:form}). 
{In particular, the objective arising from the MNL choice model, which is in the form of fractional programming, is non-linear,
because the probability of choosing a particular product from an offered assortment is the product's utility normalized (divided) by utility of all the other products in the offered assortment,
{where utilities are fixed parameters and in an MNL choice model they correspond to the exponentiated average of Gumbel valuations.}
In a similar vein, resource constraints involving the MNL choice model are also non-linear because they involve purchase probabilities of offered assortments that are expressed as fractions. }
In most existing research on re-solving, the linearity of objectives or constraints of fluid approximation plays an important role in both computation and analysis. To address the challenge brought by the non-linearity, we develop a new epoch-based re-solving algorithm. The technical novelty in our algorithm is two-fold:
\begin{enumerate}
\item Instead of re-solving the original fluid approximation, we propose a new linear programming (LP) formulation that moves the denominator of the fluid approximation into the constraint. This re-formulation not only has computational advantages, but also greatly facilitates regret analysis. By leveraging the envelop theorem \citep{milgrom2002envelope,tercca2020envelope}, this re-formulation greatly helps the analysis of the stability of the re-solved problem.

\item To determine the right time for resolving, we adopt an epoch-based approach. In particular, our algorithm first resolves the original fluid approximation to estimate the number of epochs. For each epoch, we re-solve the LP and sample an assortment based on the LP solution. The sampled assortment will be offered repeatedly until no-purchase, which concludes this epoch.
The idea of offering the same assortment until no-purchase is originated from \cite{agrawal2019mnl}. However, the use of such an epoch-based strategy is for a different reason in our problem. Due to the non-linearity of the objective, if one samples an assortment at each time period,  the probability of making a purchase based on the sampled assortment will \emph{not} equal to the objective of the fluid approximation. On the other hand, the epoch-based sampling until no-purchase will guarantee the equality holds.

\item The optimal solution to the fluid approximation is in the form of a probability distribution. Converting such a distribution to actual assortment offerings is non-trivial, because hard 
resource capacity constraints must be satisfied. In this paper, we use an efficient sampling algorithm based on the Birkhoff-von-Neumann decomposition to carry out the sampling
efficiently while at the same time respecting hard capacity constraints.
\end{enumerate}

With this epoch-based re-solving technique in place,  the main theoretical contribution of the paper is the development of the total regret, i.e., the total gap of expected profit between fluid approximation solution and our policy. 
 We proved that the total regret is logarithmic in the time horizon $T$.
 Prior to this paper, the works on computationally efficient methods for assortment optimization with knapsack constraints have regret on the order of $\widetilde O(\sqrt{T})$, by utilizing primal dual approaches \citep{jiang2020online,miao2021general} or solving fluid approximation with $NT$ optimization variables \citep{lei2018randomized}, both of which are quite different from the techniques used in this paper.

We also briefly mention that there is some recent literature on dynamic assortment optimization on inventory constraints (see, e.g., \citet{Golrezaei2014,Chen:16recom,Fata19mutli} and references therein).  These works assume different customer types with adversarial arrivals and focus on developing competitive ratios for the worst-case performance. Under the competitive ratio criteria, recent works \citep{Gong:19:online, Feng:20:near} further study the impact of reusable products in online assortment optimization. There is another line of research that assumes that the underlying choice model is unknown (note that we assume the utility parameters of the underlying MNL model are given). The goal is to learn the choice model and make the assortment selection simultaneously (see, e.g., \citet{rusmevichientong2010dynamic, Saure2013,agrawal2017thompson,Chen:18tight, agrawal2019mnl,Chen:20context,chen2021dynamic,chen2021optimal}). This line of research usually does not take the inventory constraint into consideration.
Furthermore, when consumers' demand model is unknown and needs to be learnt, it is theoretically not possible to achieve regret lower than $\Omega(\sqrt{T})$,
as shown by the works of \cite{agrawal2019mnl,Chen:18tight}.
This makes the learning-while-doing setting not directly comparable with the results derived in this paper. 

{
Regarding specifically results obtaining \emph{competitive ratios}, \cite{Golrezaei2014} studied online assortment planning with multple consumer classes whose arrival patterns and preferences are adversarial, and derived an $1-1/e$ competitive ratio for large inventory settings. With additional assumptions of stochastic arrivals, the competitive ratios can be further improved to 0.75 or 0.78. \cite{Gong:19:online} extends the method to reusable resources settings with random resource consumption times,
improving over an earlier result of $1/2$ competitive ratio established in \citep{rusmevichientong2020dynamic} for similar settings.
\cite{goyal2025asymptotically} shows that when resource consumption durations are unknown, previously known methods are sub-optimal
and a new policy must be designed to achieve the optimal $1-1/e$ competitive ratio.
We remark that, in all of the works mentioned in this paragraph, there is something inherently \emph{adversarial} in the model such as arriving patterns or preference parameters in choice models, rendering competing against the best policy in hindsight impossible even with asymptotically long time horizons.
In contrast, this work studies \emph{stochastic} purchase behaviors of homogeneous consumers with a long time horizon,
and therefore the ``equivalent'' competitive ratio could approach 1 as $T$ goes to infinity, and we are precisely characterizing how fast the performance of our propose policy approaches that of the optimal policy in hindsight.
}
 
The rest of the paper is organized as follows. Sec.~\ref{sec:form} introduces the background of the problem, including the problem setup, fluid approximations, and necessary technical assumptions. Sec.~\ref{sec:policy} describes the proposed re-solving policy. Sec.~\ref{sec:regret} provides the regret analysis and presents our main theoretical result in Theorem \ref{thm:constant-regret}. 
Numerical results are given in Sec.~\ref{sec:numerical}. We conclude the paper in Sec.~\ref{sec:conclusion}. 

\section{Problem Formulation and Preliminaries}
\label{sec:form}

We study a stylized multi-period assortment optimization problem with knapsack constraints.
At the beginning, the retailer has access to $M$ types of resources. Let us denote  initial inventory for the type-$j$ resource by $C_{0,j}>0$ for $j=1,\ldots, M$.
The retailer offers $N$ different types of substitutable products, with each unit of product $i$ consuming $A_{ij}$ units of inventory of resource $j$.
The entire dynamic assortment optimization problem consists of $T$ time periods.
At the beginning of time period $t$, the retailer offers an \emph{assortment} $S_t\subseteq[N]$ of size at most $K$,
such that sufficient inventory levels are present for each single product in the offered assortment (i.e., $C_{t,j}\geq A_{i,j}$ for all $i\in S_t$ and $j\in[M]$).
The customer then either chooses to purchase one of the presented product $i_t\in S_t$, or leaves without purchasing anything (denoted as $i_t=0$).

If a product $i\in S_t$ is purchased,  the retailer collects profit $r_i\in[0,1]$ and the inventory level of resource type $j\in[M]$ is depleted by an amount of $A_{ij}$; if no product is purchased then the retailer collects no revenue and no inventory is depleted. In this paper, we use $y_t\in[0,1]$ to denote the retailer's profit at time $t$.
  {For notational simplicity, we denote by the $M$-dimensional vector $C_t=(C_{t,1},\ldots, C_{t,M})$ the remaining inventory for all $M$ resources at the beginning of time $t$ for $t=0,\ldots, T$. We also denote by $A_{i,\cdot}=(A_{i,1},\ldots, A_{i,M})$ the resource consumptions for each unit of product $i$.  Thus, if the product $i$ is purchased during the period $t$, the remaining inventory at beginning of time $t+1$ will be $C_{t+1}=C_t-A_{i,\cdot}$.	
}

To model the consumers' randomized discrete choices we use the classical multi-nomial logit (MNL) choice model.
The MNL model associates each product type $i\in[N]$ with a preference score  $v_i>0$ (i.e., the exponential of the mean utility for the $i$-th product).
Given an assortment $S_t\subseteq[N]$, the customer's purchasing decision $i_t$ is governed by 
\begin{equation}
\Pr[i_t=i|S_t] = \frac{v_i}{v_0 + \sum_{k\in S_t}v_k},\;\;\;\;\;\;i\in S_t\cup\{0\},
\label{eq:defn-mnl}
\end{equation}
where $v_0 \equiv 1$ in the denominator denotes the preference score of ``no-purchase''. 
The expected revenue of an assortment $S_t$ can then be expressed as
\begin{equation}
R(S_t)=\sum_{i\in S_t}r_i\Pr[i_t=i|S_t] =\frac{\sum_{ i \in S_t }r_i v_i}{v_0 + \sum_{i \in S_t}v_i}.
\label{eq:defn-R}
\end{equation}

The objective of the retailer is to design a dynamic assortment optimization policy so that the expected profit is maximized subject to resource inventory constraints.
More specifically, an admissible policy $\pi$ consists of $T$ mappings $\pi_1,\pi_2,\cdots,\pi_T$,
with $\pi_t$ mapping the inventory levels $C_{t,1},\cdots,C_{t,M}$ at the beginning of time $t$ to an assortment selection $S_t\subseteq[N]$, $|S_t|\leq K$.
The expected total revenue for a policy $\pi$ with initial inventory vector $C_0$
can then be expressed as $V_1^{\pi}(C_0)=\mathbb E^\pi[\sum_{t=1}^T r_{i_t}]$, with the understanding that $r_0=0$.
Note that throughout the assortment optimization procedure, the retailer has full knowledge of the model $\mathcal M=(A_{ij},v_i,r_i)_{i,j=1}^{N,M}$.

{
}

\subsection{Fluid approximation and fractional relaxations}

In this section, we introduce a fluid approximation of the original DP problem. 
First, we introduce the following \emph{encoding scheme} of an assortment. For a given assortment $S$, we encode $S$ by a vector $x$ of length $N$ with each element $x_i=\vct 1\{i\in S\}$ for $i=1,\ldots, N$. Here, the expression $\vct 1\{i\in S\}$ denotes the indicator function, which takes the value 1 when $i \in S$ and the value 0 otherwise. For the computational tractability of the fluid approximation, we adopt a fractional relaxation approach by allowing $x$ to be a real-valued vector, i.e., $x \in [0,1]^N$. This relaxation enables us to extend the domain of $R(\cdot)$ to $[0,1]^N$
as follows:
\begin{equation}
R(x) := \frac{\sum_{i=1}^N r_iv_ix_i}{1+\sum_{i=1}^N v_ix_i}, \;\;\;\;\;\;\forall x\in[0,1]^N.
\end{equation}
Note that for any $S\subseteq[N]$, the value of $R(x)$ with $x_i=\vct 1\{i\in S\}$ exactly matches the definition of $R(S)$ in Eq.~(\ref{eq:defn-R}).
We then introduce the fluid approximation formulation below, which maximizes $R(x)$ over $x\in[0,1]^N$ with the inventory constraints relaxed so that they hold only in expectation:
\begin{definition}[Fluid approximation]
For a given model $\mathcal M=(A_{ij},r_i,v_i)_{i,j=1}^{N,M}$ and initial inventory vector $C_0\in\mathbb R_+^M$,
define $\gamma_{0,j} := C_{0,j}/T$ for all resource type $j\in[M]$ and $\gamma_0=(\gamma_{0,1},\cdots,\gamma_{0,M})\in\mathbb R_+^M$.
The fluid approximation problem is formulated as
\begin{eqnarray}
x^F &:=& \arg\max_{x\in[0,1]^N} R(x)\;\;\;\;\;\;s.t.\;\; \|x\|_1\leq K,\;\; A^\top \nu(x) \leq \gamma_0,
\label{eq:fluid}
\end{eqnarray}
where $[\nu(x)]_i = v_ix_i/(1+\sum_{k=1}^N v_kx_k)$.
\label{defn:fluid}
\end{definition}

{
It should be noted that the fluid approximation introduced in Definition \ref{defn:fluid} is ``static'' in nature because it is independent of the actual planning horizon $T$,
which is in contrast to certain ``dynamic'' approximation benchmarks such the ``hindsight optimum'' benchmark \citep{Wang2020resolve}
and the dynamic programming benchmark \citep{Wang2020constant} for network revenue management.
{Note that Eq.~(\ref{eq:fluid}) involves the quantity $\gamma_0$ obtained by dividing the initial inventory levels $C_0$ by the time horizon $T$, because it represents the target resource consumption ``per period''.}
An interesting future direction of research would be
to obtain more powerful constant regret bounds or to remove non-degeneracy conditions such as those listed in Assumption (A3), 
using such alternative approximation benchmarks.
}

{
The following lemma shows that Definition \ref{defn:fluid} serves as an upper bound to the expected revenues of any admissible policies,
after normalization across time horizon is properly carried out.
\begin{lemma}
For any admissible policy $\pi$ over $T$ time periods, it holds that $\mathbb E^\pi[\sum_{t=1}^T R(S_t)] \leq T R(x^F)$.
\label{lem:fluid}
\end{lemma}
While Lemma \ref{lem:fluid} looks standard, its proof is highly non-trivial because we can no longer apply arguments involving Jensen's inequality
(which is how similar results on network revenue management are typically proved; see e.g.~\citep{jasin2014reoptimization,Wang2020constant})
because the objective function $R(\cdot)$ in Eq.~(\ref{eq:fluid}) is \emph{not} concave.
Instead, we construct two auxiliary LP formulations, one based on probability distributions over \emph{all} possible assortments and the other 
based on the Charnes-Cooper transform of Eq.~(\ref{eq:fluid}), and show that they share the same dual problem as Eq.~(\ref{eq:fluid}),
yielding the same optimal values that upper bound the value of any admissible policy $\pi$.

To avoid interrupting the flow of this section, we place the complete proof of Lemma \ref{lem:fluid} in the appendix.
}

\begin{remark}
While $R(x^F)$ upper bounds the per-period expected reward of any admissible policy, it is in general not true that $\mathbb E[R(z)]=R(x^F)$ if $\mathbb E[z]=x^F$,
meaning that implementing a policy that uses i.i.d.~assortment samples per period is likely sub-optimal in the long run.
To correct for such bias, we use the strategy of offering assortments \emph{repetitively until the first no-purchase activity}, which converges to $R(x^F)$ in the long run.
Further details are given in Algorithm \ref{alg:resolve} and Sec.~\ref{subsec:epoch} later this paper.
\end{remark}

While Eq.~(\ref{eq:fluid}) has rather complex structures, it should be noted that the objective function $R(x)$ is a fractional of two linear functions of $x$, and the constraint $A^\top\nu(x)\leq\gamma_0$
can be re-formulated so that it is a linear constraint in $x$. The fluid approximation can then be efficiently solved by using fractional programming techniques, as we remark in Sec.~\ref{subsec:fluid-solve} later.


Note that, when the model $\mathcal M=(A_{ij},v_i)_{i,j=1}^{N,M}$ and the capacity constraint $K$ is fixed,
the solution and objective of the fluid approximation (\ref{eq:fluid}) only depends on the normalized inventory vector $\gamma_0\in\mathbb R_+^M$.
Thus, we define $\Phi(\gamma)$ as the optimal objective of the fluid approximation when the normalized inventory vector corresponds with $\gamma$, or more specifically
\begin{equation}
\Phi(\gamma) := \max_{x\in[0,1]^N} R(x)\;\;\;\;\;\;s.t.\;\; \|x\|_1\leq K,\;\; A^\top \nu(x) \leq \gamma.
\label{eq:Phi}
\end{equation}
%

In this paper we also study a slightly different fluid approximation. For a normalized inventory vector $\gamma\in\mathbb R_+^M$
and a denominator budget parameter $s\geq 1$, define
\begin{equation}
\textstyle
\Psi(\gamma,s) := \max_{x\in[0,1]^N} \sum_{i=1}^N r_iv_ix_i \;\;\;\;\;\;s.t.\;\;\|x\|_1\leq K,\;\; A^\top\nu(x)\leq\gamma,\;\; 1+\sum_{i=1}^Nv_ix_i \leq s.
\label{eq:Psi}
\end{equation}
Comparing $\Psi(\gamma,s)$ with $\Phi(\gamma)$, we note that $\Psi(\gamma,s)$ can be converted into a linear program,
making it easier to solve and analyze. The connection between $\Phi$ and $\Psi$ is made explicit by the following lemma, which is proved in the appendix (see Appendix \ref{app:proof}).
\begin{lemma}
{For any $\gamma\geq 0$,}
let $x^F$ be the optimal solution to $\Phi(\gamma)$ and $s=1+\sum_{i=1}^N {v_i}x_i^F$. Then $\Phi(\gamma)=R(x^F) = \Psi(\gamma,s)/s$.
\label{lem:phi-psi}
\end{lemma}

\subsection{Assumptions}

We make the following assumptions throughout this paper.
\begin{enumerate}
\item[(A1)] There exists finite constants $A_0,B_0<\infty$ such that $v_i\in[0,B_0]$ and $A_{ij}\leq A_0$ for all $i\in[N]$, $j\in[M]$;
\item[(A2)] There exists a positive constant $\gamma_{\min}>0$ such that $\gamma_{0,j}\geq\gamma_{\min}$ for all $j\in[M]$;
\item[(A3)] {Recall the definition that $x^F$ is the optimal solution to the fluid approximation problem in Eq.~(\ref{eq:fluid}).}
There exist positive constants $0<\rho_0,\chi_0<\gamma_{\min}$ such that, for any $\delta\in\mathbb R^M,\epsilon\in\mathbb R$ satisfying $\|\delta\|_{\infty}\leq\rho_0$ and
$|\epsilon|\leq\rho_0$, the following hold:
(1) $\sum_i v_ix_i^F> \rho_0$; (2) the set of binding/active constraints in $\Psi(\gamma_0+\delta,1+\sum_iv_ix_i^F+\epsilon)$ is the same as $\Psi(\gamma_0,1+\sum_i v_ix_i^F)$;
(3) after converting the $(M+2)$-linear constraints and $2N$ bounded interval constraints $0 \le x_i\le 1$ of $\Psi$ into the standard form $\tilde Ax\leq \tilde b$ for some numerical matrix $\tilde A\in \R^{(M + 2N + 2)\times N}$, the sub-matrix of $\tilde A$ which corresponds to the binding constraints in $\Psi(\gamma_0+\delta,1+\sum_iv_ix_i^F+\epsilon)$ is invertible with its smallest singular value being greater than or equal to $\chi_0$.  
\end{enumerate}

Assumption (A1) assumes that the product utility parameters and the amount of resources consumed by each product sale are upper bounded.
{Assumption (A2) assumes that the initial inventory levels $C_{0,j}$ of all resources scale linearly with time horizon $T$; i.e.~$C_{0,j}\geq \gamma_{\min}T$ for all $j$,
which is equivalent to $\gamma_{0,j}\geq\gamma_{\min}$ where $\gamma_{0,j}=C_{0,j}/T$ is the normalized inventory level for resource type $j$.}
Both assumptions are standard in the revenue management with inventory constraints literature \citep{WDY2014,ferreira2018online,besbes2012blind,jasin2014reoptimization}.

Assumption (A3) is a non-degeneracy condition which needs more explanations.
In essence, it assumes that the set of binding/active constraints in the re-formulated fluid approximation $\Psi(\gamma,s)$ is \emph{stable}
with respect to the input parameters, in the sense that a small perturbation of $\gamma$ or $s$ leads to the same set of binding/active resource constraints.
It also assumes that the $M+ 2N + 2$ linear constraints in the LP formulation of $\Psi$ are ``in general position'', in the sense that any of its square sub-matrices are invertible
so that no row is a linear combination of fewer than $(N-1)$ other rows/columns in the constraint matrix.
At a higher level, this assumption is similar to the ``non-degenerate'' assumptions in the network revenue management literature \citep{jasin2012re,Wang2020resolve,wu2015algorithms}, and people have shown that without certain non-degenerate conditions the re-solving heuristic does not
necessarily work \citep{cooper2002asymptotic}. 
{
In particular, in our analysis Assumption (A3) implies that $\Psi$ is twice differentiable with respect to $\gamma$ and $s$, a key property
that is essential to our proof of an important Lemma \ref{lem:hessian} in the appendix.
}
Overall, we leave the question of achieving logarithmic or even constant regret without degenerate conditions 
like Assumption (A3) as an interesting future direction, which likely requires new algorithms and analytical insights.

\section{Re-solving Policy}
\label{sec:policy}

In this section, we present a re-solving policy that achieves a logarithmic regret under the assumptions (A1) through (A3).
We start with two building blocks of the policy: to efficiently solve the fluid approximation problem,
and to sample an assortment with its mean dictated by the fluid solution $x^F$ and subject to hard capacity constraints.
We then present the proposed re-solving policy, and explain the design principles behind the policy.
Regret analysis is presented in full details in the next section.


\subsection{Solving the fluid approximation problem}\label{subsec:fluid-solve}

\begin{algorithm}[t]
\caption{Solving the fluid approximation problem}
\label{alg:fluid}
\begin{algorithmic}[1]
\State \textbf{Input}: $\gamma\in\mathbb R_+^N$, model $\mathcal M=(A_{ij},v_i)_{i,j=1}^{N,M}$, capacity constraint $K$, error tolerance $\epsilon$.
\State \textbf{Output}: {a solution $\hat x^F\in[0,1]^N$ feasible to all constraints in $\Phi(\gamma)$.}
\State Initialize: $\lambda_L=0$, $\lambda_U=1$, $\hat x^F=0$.
\While{$\lambda_U-\lambda_L\geq\epsilon$}
	\State Let $\lambda = (\lambda_L + \lambda_U)/2$ and solve the following linear programming problem:
	\begin{equation}
	 x_\lambda^F = \arg\max_{x\in[0,1]^N}\sum_{i=1}^N (r_i-\lambda)v_ix_i \;\;s.t.\;\; \|x\|_1\leq K,\;\;\sum_{i=1}^N A_{ij}v_ix_i \leq \gamma_j\left[1+\sum_{i=1}^N v_ix_i\right],\;\forall j\in[M]; 
	\label{eq:lambda-fluid}
	\end{equation}
	\State If $\sum_{i=1}^N (r_i-\lambda)v_i[x_\lambda^F]_i \geq \lambda$ then $(\lambda_L,\lambda_U)\gets (\lambda,\lambda_U)$, $\hat x^F=x_\lambda^F$;
	else $(\lambda_L,\lambda_U)\gets (\lambda_L,\lambda)$;
\EndWhile
\end{algorithmic}
\end{algorithm}

We first study the question of solving the fluid approximation $\Phi(\gamma)$ efficiently.
Note that $\Phi(\gamma)$ is \emph{not} a convex optimization problem: the objective function,
being the ratio of two linear functions of optimization variables, is \emph{quasi-concave} instead of concave.
This makes classical first-order convex optimization techniques difficult to apply.
Furthermore, the constraint sets are convex but it involves many constraints and therefore projection operators are difficult to design.
Therefore, it is challenging to apply standard (projected) first-order optimization methods to solve $\Phi(\gamma)$.

In this section we show that, with some simple transforms the fluid approximation can be reduced to a series of linear programming problems,
which can be subsequently solved efficiently using existing optimization software or packages.

The pseudocode description of the algorithm that solves the fluid approximation problem is given in Algorithm \ref{alg:fluid}. The $\lambda$ parameter in the algorithm serves as a lower bound of the objective function. The idea of bi-section search of $\lambda$ is due to the work by \cite{rusmevichientong2010dynamic},
whose ideas can be traced back to fractional programming works by \cite{Megiddo1979}.
More specifically, it is a simple observation that for any $\lambda\in[0,1]$, the optimal objective value in $\Phi(\gamma)$ is larger than or equal to $\lambda$
if and only if the optimal objective value in Eq.~(\ref{eq:lambda-fluid}) is larger than or equal to $\lambda$,
because for any feasible $x$ we have that {$\frac{\sum_i r_iv_ix_i}{1+\sum_i v_ix_i}\geq \lambda\Longleftrightarrow \sum_i (1-\lambda)r_iv_ix_i\geq\lambda$}.
Consequently, we use a binary search approach to locate the largest $\lambda$ value for which the objective in Eq.~(\ref{eq:lambda-fluid}) is at least $\lambda$,
which is then the optimal solution to the original $\Phi(\gamma)$ problem. The formal proof of Lemma \ref{lem:fluid-solver} is placed in the appendix (see Appendix \ref{app:proof}).

\begin{lemma}
For any $\epsilon\in(0,1)$, Algorithm \ref{alg:fluid} returns {a solution $\hat x^F\in[0,1]^N$ feasible to all constrains imposed in $\Phi(\gamma)$} satisfying $R(\hat x^F)\geq \Phi(\gamma)-\epsilon$
by solving $O(\log(1/\epsilon))$ LPs with $N$ variables and $(M+1)$ constraints.
\label{lem:fluid-solver}
\end{lemma}

Because the convergence of Algorithm \ref{alg:fluid} is very fast (logarithmic in the inverse of precision parameter $\epsilon$), in the rest of the paper we shall assume that Algorithm \ref{alg:fluid}
returns exactly the optimal fluid solution $x^F$ such that $R(x^F)=\Phi(\gamma)$.

\subsection{Sampling from the fluid solution}

\begin{algorithm}[t]
\caption{Sampling from a fluid solution with hard capacity constraints}
\label{alg:sampling-fluid}
\begin{algorithmic}[1]
\Function{SampleAssortment}{$x,K$} \Comment{$x\in[0,1]^N$, $\|x\|_1\leq K$}
	\State Extend $x\in[0,1]^N$ to $x\in[0,1]^{N+K}$ such that $x_{N+1}=\cdots=x_{N+K}=1-\|x\|_1/K$; \label{line:extend}
	\State Construct a doubly stochastic matrix $\mat M\in\mathbb R^{(N+K)\times (N+K)}$ as $\mat M_{ij}=x_j/K$ if $i\in\{1,2,\cdots,K\}$, and 
	$\mat M_{ij}=(1-x_j)/N$ if $i\in\{K+1,K+2,\cdots,N+K\}$;
	\State Compute the \emph{Birkhoff-von-Neumann} decomposition of $\mat M$:
	$
	\mat M=\sum_{\ell=1}^L\alpha_\ell \mat P_\ell,
	$
	where $\alpha_\ell\in[0,1]$, $\sum_\ell\alpha_\ell=1$ and $\{\mat P_\ell\}_{\ell=1}^L$ are permutation matrices;
	\State Sample $u\in[L]$ from the multinomial distribution with parameters $\alpha_1,\cdots,\alpha_L$;
	\State \textbf{return} $z\in\{0,1\}^N$ with $z_i=1$ if and only if $\sum_{k=1}^K\mat P_{u, ki} =1 $ for each $i \in [N]$. \label{line:return}
\EndFunction
\end{algorithmic}
\end{algorithm}

In Algorithm \ref{alg:sampling-fluid} we present a computationally efficient routine \textsc{SampleAssortment} to obtain a random assortment sample
whose expected value is equal to a given fractional vector $x\in[0,1]^N$.
Because the sampled assortment must satisfy a hard capacity constraint, it cannot be obtained by simply drawing independent product samples.
Instead, we use a method based on the Birkhoff-von-Neumann decomposition of doubly stochastic matrices to implement the weighted sampling procedure.
Similar method has been used for the classical question of \emph{fair random assignment} \citep{budish2013designing}, a problem
that is similar in nature to the sampling question we have.

Below we explain the ideas of Algorithm \ref{alg:sampling-fluid} in more details. First, we manipulate the obtained solution $x$ (with $\|x\|_1 \leq K$) such that the cardinality constraint is always tight, i.e.,  $\|x\|_1=K$. To this end, we construct $K$ additional ``dummy'' choices corresponding to ``no product'', each with weight $1-\|x\|_1/K$ (see Line \ref{line:extend} in Algorithm \ref{alg:sampling-fluid}).
Then, the key step is to construct a doubly stochastic matrix $\mat M$ (i.e., a matrix whose elements are in $[0,1]$, with every row and column of the matrix summing to one)
where the first $K$ rows of $\mat M$ are $x/K$, and the last $N$ rows of $\mat M$ are $(1-x)/N$. It is easy to verify that $\mat M$ is an $(N+K)\times (N+K)$
doubly stochastic matrix.  
We then use the Birkhoff-von-Neumann algorithm (see, e.g., \cite{dufosse2016notes}) to decompose the matrix $\mat M$ into a weighted
sum of \emph{permutation matrices} $\mat P_1,\cdots,\mat P_L$:
$$
\mat M = \alpha_1\mat P_1 + \cdots + \alpha_L\mat P_L,
$$
where each $\mat P_\ell$ corresponds to a permutation $\sigma_\ell: [N+K]\to[N+K]$ as $\mat P_{\ell,ij} = \vct 1\{i=\sigma_\ell(j)\}$,
{and $\alpha_1,\cdots,\alpha_L\in[0,1]$, $\alpha_1+\cdots+\alpha_L=1$ forms a convex combination of the involved permutation matrices.}
It is known that $\mat M$ can be decomposed into at most $L=(N+K)^2$ permutation matrices, with the decomposition being computed in $O((N+K)^4)$ time.
The algorithm then samples a permutation matrix $\mat P_\ell$ with probability $\alpha_\ell$, and reads out the first $K$ rows of $\mat P_\ell$
as the selected products in the returned assortment (see Line \ref{line:return} in Algorithm \ref{alg:sampling-fluid}).
A rigorous proof of Proposition \ref{prop:sampling-fluid} is given in the appendix (see Appendix \ref{app:proof}).

\begin{proposition}
Let $z\in\{0,1\}^N$ be the output of Algorithm \ref{alg:sampling-fluid}. Then $\|z\|_1\leq K$ almost surely and $\mathbb E[z_i]=x_i$ for all $i\in[N]$.
\label{prop:sampling-fluid}
\end{proposition}

\begin{remark}
While directly applying the Birkhoff-von-Neumann algorithm leads to $L\leq (N+K)^2$ permutation matrices and $O((N+
K)^4)$ time complexity for decomposing
$\mat M$, it is possible to reduce the time complexity by utilizing the special structures of the doubly stochastic matrix $\mat M$ constructed in our problem.
In Appendix \ref{appsec:numerical} we present a reduced form of the Birkhoff-von-Neumann algorithm that decomposes $\mat M$ into $L\leq N+K$ permutation matrices
with time complexity $O((N+K)^2)$.
\end{remark}

\subsection{An epoch-based algorithm with re-solving}\label{subsec:epoch}

\begin{algorithm}[t]
\caption{An epoch-based algorithm with re-solving heuristic}
\label{alg:resolve}
\begin{algorithmic}[1]
\State \textbf{Input}: $\mathcal M=(A_{ij},v_i)_{i,j=1}^{N,M}$, $\gamma_0\in\mathbb R_+^M$, $T$, $K$.
\State Invoke Algorithm \ref{alg:fluid} with $\gamma_0,\mathcal M,K,\epsilon=1/T$ to obtain $x^F\in[0,1]^N$; let $\tau_0 = T/(1+\sum_iv_ix_i^F)$;
\For{epochs $\tau=\tau_0,\tau_0-1,\cdots, 1$ or until $T$ periods elapsed/inventory levels depleted}
	\State {Let $t^\tau\in\{0,1,2,\cdots,T\}$ be the number of time periods that are remaining at the beginning of epoch $\tau$,} and $\gamma^\tau=C_{T-t_\tau+1,\cdot}/t^\tau\in\mathbb R_+^M$ be the inventory vector
	normalized by $t^\tau$; let $s^\tau=t^\tau/\tau$;
	\State Let $x^\tau$ be the optimal solution to $\Psi(\gamma^\tau,s^\tau)$, or more specifically
	\begin{eqnarray}
	x^{\tau} &:=& \arg\max_{x\in[0,1]^N} \sum_{i=1}^N r_iv_ix_i;s.t. \|x\|_1\leq K, 1+\sum_{i=1}^Nv_ix_i \leq s^\tau, \sum_{i=1}^N A_{ij}v_ix_i \leq \gamma^\tau_j\left(1+\sum_{i=1}^Nv_ix_i\right).\label{eq:fluid-resolve}
	\end{eqnarray}
	\State Obtain $z^\tau = \textsc{SampleAssortment}(x^\tau,K)$;
	\State Offer assortment $S^\tau=\{i\in[N]: z^\tau_i = 1\}$ repeatedly until a no-purchase action occurs;
\EndFor
\end{algorithmic}
\end{algorithm}

Based on the fluid solution $x^F$ and the \textsc{SampleAssortment}($x,K$) sub-routine, we are ready to describe an epoch-based policy
with re-solving heuristics in Algorithm \ref{alg:resolve}.

The first step of Algorithm \ref{alg:resolve} is to compute the fractional solution $x^F$ to the fluid approximation $\Phi(\gamma_0)$.
{The algorithm then computes $\tau_0=T/(1+\sum_i v_ix_i^F)$, which is the expected number of epochs corresponding to $x^F$
if \emph{hypothetically} $x^F$ is used throughout the entire $T$ time periods,
because sampling assortments with respect to $x^F$ leads to $\sum_i v_ix_i^F$ number of consecutive purchases in expectation before the first no-purchase activity,
 as observed in \citep[Corollary A.1]{agrawal2019mnl}.}
\footnote{When assortment $S$ is repeatedly offered, the number of time periods before the first no-purchase activity (inclusive)
is a Geometric random variable with mean $1+\sum_{i\in S}v_i$ because for every $k\geq 1$, $\Pr[n=k|S]=(\frac{\sum_{i\in S}v_i}{1+\sum_{i\in S}v_i})^{k-1} \frac{1}{1+\sum_{i\in S}v_i}$.
This observation is first made in the work of \cite{agrawal2019mnl}.
When excluding the no-purchase time period, the (expected) number of time periods is then $\sum_{i\in S}v_i$, whose expectation is $\sum_i v_ix_i^F$ because sampling of assortments is unbiased with respect to $\{x_i^F\}_i$.}

Afterwards, Algorithm \ref{alg:resolve} operates in epochs, separated by no-purchase activities. 
{More specifically, an \emph{epoch} consists of consecutive time periods starting with a time period immediately following one during which no purchase is made, and ending with the first time period during which no purchase is made by the arriving customer.}
After each no-purchase activity and the end of an epoch, 
the algorithm re-solves a fluid approximation, using the remaining number of inventories, time periods and epochs as input parameters.
Note that the fluid approximation used in the resolving procedure $\Psi(\gamma^\tau,s^\tau)$ has different formulation from the fluid approximation $\Phi(\gamma_0)$
solved at the beginning of the algorithm, and $\Psi(\gamma^\tau,s^\tau)$ is a standard linear program.
The reason for using a different fluid approximation formulation in the resolving steps is given in the next paragraph.
After a new updated fractional solution $x^\tau$ is obtained, the assortment to be offered for the next few customers is sampled
from $x^\tau$ using the sampling strategy outlined in Algorithm \ref{alg:sampling-fluid}.

Compared with many existing re-solving heuristics (see, e.g., \cite{jasin2012re,jasin2014reoptimization}), there are two noticeable 
unique properties of Algorithm \ref{alg:resolve}: that the re-solving optimization problems are different from the fluid approximation,
and that each sampled assortment is offered repeatedly.
We remark on the motivations of the two key design ideas as follows:
\begin{enumerate}
\item The primary motivation for using a different optimization problem in Eq.~(\ref{eq:fluid-resolve}) from the direct fluid approximation $\Phi(\cdot)$
is to ensure that the re-solving problem is a standard linear programming (LP).
This not only has computational advantages, but also offers an effective way to analyze the \emph{stability} of the re-solved problem through the envelope theorem.
To ensure that the LP problems are good approximation of the original fluid problem $\Phi(\cdot)$, we transform the denominator in the $\Phi(\cdot)$ objective
as a standalone constraint $1+\sum_iv_ix_i\leq s^\tau$ to make sure that the obtained continuous solution $x^\tau$ will not have too large value in the denominator of $R(x)=\frac{\sum_i v_ix_i}{1+\sum_i v_ix_i}$.
\item The ``epoch-based approach'' (i.e., offering the same assortment $S^\tau$ repeatedly until the first no-purchase action occurs) was proposed
and analyzed in the seminal work of \cite{agrawal2019mnl} to facilitate convenient estimates of parameters in an MNL choice model.
More specifically, with an epoch-based approach the expected revenue during an epoch has a simple linear form in the product utility parameters,
leading to more convenient theoretical analysis.
In Algorithm \ref{alg:resolve}, the epoch-based approach is adopted for a different motivation, to get around the non-linearity of $R(\cdot)$.
More specifically, if one samples a random assortment $z$ for every \emph{single} time period, the {probability of making a purchase when offering} $z$ will \emph{not} equal
the objective of the fluid approximation (that is, $\mathbb E_{z\sim x}[R(z)]\neq R(\mathbb E_{z\sim x}[z]) = R(x)$).
By using the epoch-based approach, we decouple the numerator and denominator in the purchasing probability function $R(\cdot)$ and utilizing
the linearity of $x$ in both the numerator and denominator of $R(\cdot)$.
\end{enumerate}

Finally, we remark on the computational efficiency of Algorithm \ref{alg:resolve}. The main computational bottleneck is to re-solve an LP with $N$ variables and $M+2$ constraints every epoch, and to sample from the fractional solution $x^\tau$ using Algorithm \ref{alg:sampling-fluid}, which takes time $O(N^4)$
if implemented naively or $O(N^2)$ with a better implementation given in the appendix.
Overall, the total time complexity of Algorithm \ref{alg:resolve} is polynomial in $N$ and linear in $T$.
This is much faster than other alternatives, which have time complexity either exponential in $N$ (e.g., solving a multi-variable dynamic programming with discretized inventory levels) or quadratic/cubic in $T$ (e.g., solving a fluid approximation with $NT$ variables, such as in the works of \cite{lei2018randomized}).

\section{Regret Analysis}
\label{sec:regret}

{
For an admissible policy $\pi$ that respects resource inventory constraints almost surely, its \emph{regret} is defined as the expected difference between the policy's cumulative revenue over $T$ time periods and 
the cumulative revenue over $T$ time periods of the optimal DP policy, which is upper bounded by 
$R(x^F)$ multiplied by $T$ thanks to Lemma \ref{lem:fluid}, where $R(x^F)$ is the optimal objective value of Eq.~(\ref{eq:fluid}).
Subsequently, for any admissible policy $\pi$, its regret is upper bounded by
$$
\mathrm{Regret}(\pi) \leq  TR(x^F) - \mathbb E^\pi\left[\sum_{t=1}^T R(S_t)\right],
$$
where $R(S_t)$ is the expected revenue of assortment $S_t$ defined in Eq.~(\ref{eq:defn-R}).
Note that, because $TR(x^F)$ upper bounds the cumulative expected revenue of \emph{any} admissible policy, including the optimal one,
the right-hand side of the above inequality is always non-negative for any admissible policy $\pi$, and a smaller regret indicates a policy with better performance.
}

The main objective of this section is to establish the following theorem that upper bounds the regret of Algorithm \ref{alg:resolve}.
\begin{theorem}
Fix model $\mathcal M=(A_{ij},v_i)_{i,j=1}^{N,M}$ and normalized initial inventory vector $\gamma_0\in\mathbb R_+^M$ that satisfy
Assumptions (A1) through (A3) with parameters $A_0, B_0,\gamma_{\min}$ and $\rho_0$.
Then for any $T$, the regret of Algorithm \ref{alg:resolve} is upper bounded by 
$$
\const.\times\frac{A_0^2B_0^7K^5M^2 N\ln T}{\min\{\rho_0^2,\chi_0^2\}},
$$
where $\const.$ is a finite, universal numeric constant.
\label{thm:constant-regret}
\end{theorem}
{
\begin{remark}
Theorem \ref{thm:constant-regret} shows that under Assumptions (A1) through (A3), the cumulative regret of Algorithm \ref{alg:resolve} scales logarithmically with the time horizon $T$
as $T$ increases to infinity.
\end{remark}
}

In the rest of this section we present the proof of Theorem \ref{thm:constant-regret}.
To facilitate the proof we first define some key notations.
For epoch $\tau$, define $n^\tau(i)$, $m^\tau$ and $\beta^\tau$ as follows:
\begin{itemize}
\item[-] $n^\tau(i)$: the number of times product $i$ is purchased during epoch $\tau$;
\item[-] $m^\tau$: the length of epoch $\tau$, or $m^\tau=1+\sum_{i=1}^Nn^\tau(i)$;
\item[-] $\beta^\tau$: the total revenue collected during the epoch $\tau$, or $\beta^\tau=\sum_{i=1}^Nr_i n^\tau(i)$.
\end{itemize}
We also recall the definitions of $\gamma^\tau,t^\tau$ and $s^\tau$ in Algorithm \ref{alg:resolve},
representing the normalized remaining inventory, the number of remaining time periods over the entire time horizon of $T$ periods, and the normalized epoch size
at the beginning of epoch $\tau$.

\subsection{Recursion of $\gamma^\tau$ and $s^\tau$}

Because the re-solved optimization problem $\Psi(\gamma^\tau,s^\tau)$ crucially depends on $\gamma^\tau$ and $s^\tau$,
it is essential to analyze the recursion and evolving of both parameters across epochs.
For a complete epoch $\tau$ (a complete epoch is one that terminates without running out of inventory or reaching the time horizon $T$, and {therefore ends exactly on a time period during which no purchase occurs}), 
define $\{\Delta^\tau(i)\}_{i=1}^N$ and $\Delta^\tau$ as
\begin{equation}
\textstyle
\Delta^\tau(i) := n^\tau(i) - v_ix^\tau_i,\;\;\;\;\;\; \Delta^\tau := \sum_{i=1}^N \Delta^\tau(i).
\label{eq:delta}
\end{equation}
The following lemma shows that $\{\Delta^\tau(i)\}_{i=1}^N$ (and subsequently $\Delta^\tau$ as well) are conditionally centered random variables.
\begin{lemma}
For every complete epoch $\tau$ and product $i\in[N]$, $\mathbb E[\Delta^\tau(i)|\gamma^\tau,s^\tau] = 0$.
\label{lem:delta-centered}
\end{lemma}
\begin{proof}{Proof of Lemma \ref{lem:delta-centered}.}
The proof is conditioned on $\gamma^\tau$ and $s^\tau$.
It is a well-known result (see, e.g., \cite[Corollary A.1]{agrawal2019mnl}) that $n^\tau(i)$ is a geometrically distributed random variable
with mean $\mathbb E[n^\tau(i)|z^\tau] = z^\tau_iv_i$.
Furthermore, by Proposition \ref{prop:sampling-fluid} we have $\mathbb E[z^\tau_i|x^\tau] = x^\tau_i$.
Lemma \ref{lem:delta-centered} is thus proved using the law of total expectation. $\square$
\end{proof}

The following lemma is the key recursion formula for $\gamma^\tau$ and $s^\tau$:
\begin{lemma}
Let $\tau>1$ be a complete epoch. Then
$$
s^{\tau-1} \geq s^\tau - \frac{\Delta^\tau}{\tau-1},\;\;\;\;\;\; \gamma^{\tau-1}_j \geq \gamma^\tau_j- \frac{[\sum_{i=1}^NA_{ij}\Delta^\tau(i)]-\min\{\gamma_j^\tau,A_0B_0K\}\Delta^\tau}{t^{\tau-1}},\;\;\forall j\in[M].
$$
\label{lem:recursion-gamma-s}
\end{lemma}

{Lemma \ref{lem:recursion-gamma-s} is proved by painstakingly tracking the definitions and dynamics governing $s^\tau$ and $\gamma^\tau$,
and making relaxations to simplify lower bounds on these two quantities when necessary.}

\begin{proof}{Proof of Lemma \ref{lem:recursion-gamma-s}.}
{From Algorithm \ref{alg:resolve}, for any resource type $j$ and epoch $\tau$, $\gamma_j^{\tau}=C_{T-t^{\tau}+1,j}/t^{\tau}$
where $t^{\tau}$ is the remaining number of time periods at the beginning of epoch $\tau$ and $C_{t-t^{\tau}+1,j}$ is the remaining amount of inventory for resource type $j$ at the beginning of epoch $\tau$. Similarly, $\gamma_j^{\tau-1}=C_{T-t^{\tau-1}+1,j}/t^{\tau-1}$. 
Note also that $C_{t-t^{\tau}+1,j}-C_{t-t^{\tau-1}+1,j}$ is the amount of consumption of resource type $j$ during epoch $\tau$,
which is equal to $\sum_{i,j=1}^NA_{ij}n^\tau(i)$ where $A$ is the resource consumption matrix and $n^\tau(i)$ is the number of times product $i$ is purchased during epoch $\tau$. Subsequently, }
\begin{align*}
\gamma^{\tau-1}_j &= \frac{\gamma^\tau_jt^\tau - \sum_{i=1}^NA_{ij}n^\tau(i)}{t^{\tau-1}} = \frac{\gamma^\tau_j(t^{\tau-1}+m^\tau)-\sum_{i=1}^NA_{ij}n^\tau(i)}{t^{\tau-1}}\\
&= \gamma^\tau_j + \frac{\gamma^\tau_j[1+\sum_k v_kx^\tau_k] - \sum_{i=1}^N A_{ij}v_ix^\tau_i}{t^{\tau-1}} - \frac{\sum_{i=1}^NA_{ij}\Delta^\tau(i)-\gamma_j^\tau\Delta^\tau}{t^{\tau-1}}\geq \gamma^\tau_j - \frac{\sum_{i=1}^NA_{ij}\Delta^\tau(i)-\gamma_j^\tau\Delta^\tau}{t^{\tau-1}},
\end{align*}
where the last inequality holds with probability 1 because $\sum_{i=1}^N A_{ij}v_ix^\tau_i\leq \gamma^\tau_j[1+\sum_k v_kx^\tau_k]$ for every $j\in[M]$
thanks to the constraints Eq.~(\ref{eq:fluid-resolve}) in the optimization problem $\Psi$. 
When $\gamma^\tau_j\geq A_0B_0K$, we use the facts that $m_\tau\geq 1$ and $\sum_{i=1}^NA_{ij}v_ix^\tau_i \leq A_0B_0\|x^\tau\|_1 \leq A_0B_0K$
to alternatively decompose $\gamma^{\tau-1}_j$ as
\begin{align*}
\gamma^{\tau-1}_j &= \frac{\gamma^\tau_j(t^{\tau-1}+m^\tau)-\sum_{i=1}^NA_{ij}n^\tau(i)}{t^{\tau-1}}
\geq \gamma^\tau_j +  \frac{A_0B_0K m^\tau - \sum_{i=1}^NA_{ij}n^\tau(i)}{t^{\tau-1}}\\
&\geq \gamma^\tau_j + \frac{A_0B_0K [\sum_{i=1}^N n^\tau(i)]}{t^{\tau-1}}- \frac{\sum_{i=1}^NA_{ij}\Delta^\tau(i)}{t^{\tau-1}}
= \gamma^\tau_j + \frac{A_0B_0 K[\sum_{i=1}^Nv_ix_i^\tau]}{t^{\tau-1}}  + \frac{A_0B_0K\Delta^\tau}{t^{\tau-1}}-\frac{\sum_{i=1}^NA_{ij}\Delta^\tau(i)}{t^{\tau-1}}\\
&\geq \gamma^\tau_j - \frac{\sum_{i=1}^NA_{ij}\Delta^\tau(i) - A_0B_0K\Delta^\tau}{t^{\tau-1}} \geq  \gamma^\tau_j - \frac{\sum_{i=1}^NA_{ij}\Delta^\tau(i) - \min\{\gamma_j^\tau,A_0B_0K\Delta^\tau\}}{t^{\tau-1}},
\end{align*}
where the last inequality holds because the $A_0B_0K\Delta^\tau$ term in the numerator is associated with a positive sign and therefore taking a minimum between it and $\gamma_j^\tau$ does not increase the right-hand side of the inequality.
This proves the inequality on $\gamma^{\tau-1}_j$.
We next consider $s^{\tau-1}$. {By definition, $s^\tau=t^\tau/\tau$ (see Algorithm \ref{alg:resolve} for the definition of $s^\tau$).
Note also that $t^{\tau-1}=t^\tau-m^\tau$, where $m^\tau$ is the number of time periods involved in epoch $\tau$. Subsequently,
$s^{\tau-1} =\frac{t^{\tau-1}}{\tau-1} = \frac{t^\tau-m^\tau}{\tau-1}= \frac{s^\tau \tau - m^\tau}{\tau-1} = s^\tau + \frac{s^\tau-m^\tau}{\tau-1} = s^\tau + \frac{s^\tau - [1+\sum_i v_ix^\tau_i]}{\tau-1} - \frac{\Delta^\tau}{\tau-1}\geq s^{\tau} - \frac{\Delta^{\tau-1}}{\tau-1}$,}
where the last inequality holds because $1+\sum_i v_ix^\tau_i\leq s^\tau$ thanks to constraint Eq.~(\ref{eq:fluid-resolve}) in the optimization problem $\Psi$.
This proves the recursion inequality about $s^{\tau-1}$.
$\square$
\end{proof}

{Lemmas \ref{lem:delta-centered} and \ref{lem:recursion-gamma-s} show that} $\{s^\tau,\gamma^\tau\}_\tau$ are sub-martingales: $\mathbb E[s^{\tau-1}|s^\tau]\geq s^\tau$
and $\mathbb E[\gamma^{\tau-1}|\gamma^\tau] \geq \gamma^\tau$.
For convenience, we also define the following notations:
\begin{equation}
\varepsilon^\tau(j):= \frac{\sum_{i=1}^NA_{ij}\Delta^{\tau}(i)-\min\{\gamma^\tau_j,A_0B_0K\}\Delta^\tau}{t^{\tau-1}}, \;\; 
\varepsilon^{\to\tau}(j) := \sum_{\tau'=\tau}^{\tau_0}\varepsilon^{\tau'}(j), \;\;
\Delta^{\to\tau} := \sum_{\tau'=\tau}^{\tau_0}\Delta^{\tau'}.
\label{eq:defn-vareps}
\end{equation}
{Intuitively, $\Delta^\tau$ is the time-normalized deviation of expected epoch lengths at epoch $\tau$, $\varepsilon^\tau(j)$ is the time-normalized deviation of consumption of resource type $j$ at epoch $\tau$, and $\Delta^{\to\tau}$, $\varepsilon^{\to\tau}$ are cumulative deviations of corresponding quantities for all epochs prior to $\tau$.}
Note also that $\gamma^{\tau_0}=\gamma_0$ and $s^{\tau_0} = 1+\sum_i v_ix_i^F$. Lemma \ref{lem:recursion-gamma-s} then implies
that for every complete epoch $\tau$ and resource type $j\in[M]$,
\begin{equation}
\textstyle
\gamma^{\tau-1}_j \geq \gamma_{0,j} - \varepsilon^{\to\tau}(j), \;\;\;\;\;\; s^{\tau-1} \geq 1+\sum_i v_ix_i^F - \Delta^{\to\tau}.
\label{eq:tracking}
\end{equation}
For convenience we also define the martingales version of $\{s^\tau,\gamma^\tau\}_\tau$ as 
\begin{equation}
\tilde\gamma^{\tau-1}_j := \gamma_{0,j} - \varepsilon^{\to\tau}(j), \;\;\;\;\;\; \tilde s^{\tau-1} := 1+\sum_i v_ix_i^F - \Delta^{\to\tau}.
\label{eq:defn-tilde-s-gamma}
\end{equation}
It is clear from Eq.~(\ref{eq:tracking}) that $\tilde\gamma^\tau\leq\gamma^\tau$ and $\tilde s^\tau\leq s^\tau$ for all $\tau$ almost surely.

\subsection{Analysis of stopping time}

Let $\tau^\circ$ be the last complete epoch during which the inventory levels of all resources remain positive and the time horizon $T$ is not reached.
Because the dynamic assortment optimization procedure either terminates or has very poor performance after epoch $\tau^\circ$,
it is vital to upper bound the (expected value) of $\tau^\circ$ to ensure that there are enough complete epochs.
This is the objective of this section.

To upper bound $\mathbb E[\tau^\circ]$ we first define $\tau^\circ$ using the tracked quantities $\{\gamma^\tau,s^\tau\}$.
Note that if {$\gamma^\tau_j>0$} then the inventory level of resource type $j$ is not depleted at the beginning of epoch $\tau$.
Also, if $s^\tau\geq 0$ then the time horizon $T$ has not yet been reached.
Therefore, we have
$\tau^\circ = \min\{\tau: \forall\tau'\geq\tau, \gamma^{\tau'}\geq 0\wedge s^{\tau'}\geq 0\} + 1$.
To allow the usage of stability results (to be established in the next section) of $\Psi$ defined in Eq.~(\ref{eq:Psi}) earlier,
we actually study a stopping time $\tau^\sharp$ that upper bounds $\tau^\circ$ almost surely.
More specifically, define
\begin{equation}
\tau^\sharp := \min\big\{\tau: \forall\tau'\geq\tau, \|\tilde\gamma^{\tau'}-\gamma_0\|_{\infty}\leq\rho_0\wedge |\tilde s^{\tau'}-s^{\tau_0}|\leq \rho_0\big\} + 1,
\label{eq:defn-tausharp}
\end{equation}
{where $\wedge$ is the logical ``AND'' operator.}
It is easy to see that $\tau^\sharp$ is a stopping time and satisfies $\tau^\sharp\geq\tau^\circ$ almost surely, thanks to Assumption (A3)
and the fact that $\gamma^\tau\geq\tilde\gamma^\tau$, $s^\tau\geq\tilde s^\tau$ almost surely.
Note that $|\tilde s^{\tau'}-s^{\tau_0}|\leq\rho_0$ implies $s^{\tau'}\geq s^{\tau_0}-\rho_0 = 1+\sum_i v_ix_i^F-\rho_0\geq 1$ thanks to Assumption (A3).
The following Lemma gives an upper bound on $\mathbb E[\tau^\sharp]$.
\begin{lemma}
It holds that $\mathbb E[\tau^\sharp] \leq  \const.\times A_0^2B_0^4K^4M\ln T/\rho_0^2$.
\label{lem:tausharp-bound}
\end{lemma}
\begin{proof}{Proof of Lemma \ref{lem:tausharp-bound}.}
The first part of the proof is to upper bound $|\tilde s^{\tau}-s^{\tau_0}|=|\Delta^{\to\tau}|$.
Recall the definition that $\Delta^{\to\tau}=\sum_{\tau'=\tau}^{\tau_0}\Delta^{\tau'}/(\tau'-1)$, where $\Delta^{\tau'}=\sum_{i=1}^N n^{\tau'}(i)-v_ix_i^{\tau'}$.
Now fix $\tau=\tau'\geq\tau^\sharp$ and decompose $\Delta^\tau$ as $\Delta^\tau = [\sum_{i=1}^Nv_ix_i^F - \sum_{i\in S^{\tau}}v_i] + [\sum_{i\in S^{\tau}} (n^{\tau}(i) - v_i)]$. For the first term, note that both quantities are upper bounded by $BK$ almost surely. Hence,
\begin{equation}
\textstyle\mathbb E\left[\big|\sum_{i=1}^Nv_ix_i^F - \sum_{i\in S^{\tau}}v_i\big|^{2}\right]
\leq B_0^2K^2.
\label{eq:proof-tausharp-1}
\end{equation}
For the second term, note that for each $i\in S^\tau$ {($S^\tau$ is the assortment offered for all time periods in epoch $\tau$),} we have $\Pr[n^\tau(i)=n] = \frac{1}{1+v_i}\left(\frac{v_i}{1+v_i}\right)^n = (1-p_i)p_i^n$ \citep[Corollary A.1]{agrawal2019mnl}, where we define $p_i:=v_i/(1+v_i)$, implying $\mathbb E[n^\tau(i)]=v_i$. Hence, 
$\mathbb E[\big|n^\tau(i)-v_i\big|^{2}]
\leq \mathbb E[(n^\tau(i))^{2}] = (1-p_i)\sum_{n=0}^{\infty} n^{2}p_i^n
= \frac{(1+p_i)p_i}{(1-p_i)^2}=v_i(2+v_i)\leq 2B_0^2.$
Noting that $(a_1+\cdots+a_K)^2 \leq K^2(a_1^2+\cdots+a_K^2)$, we have that
\begin{equation}
\textstyle \mathbb E\left[\big|\sum_{i\in S^\tau} (n^\tau(i)-v_i)\big|^2\right] \leq 2K^2B_0^2.
\label{eq:proof-tausharp-2}
\end{equation}
Combining Eqs.~(\ref{eq:proof-tausharp-1},\ref{eq:proof-tausharp-2}) we have $\mathbb E[|\Delta^\tau|^2]\leq 2^2 \times (1 + 2) K^2 B_0^2 = 12K^2B_0^2$.
Note also that $\{\Delta^\tau\}$ is a martingale difference sequence and therefore $\mathbb E[\Delta^\tau|\Delta^1,\cdots,\Delta^{\tau-1}]=0$.
Subsequently,
\begin{align}
\mathbb E\left[\big|\Delta^{\to\tau}\big|^2\right]
&= \sum_{\ell=\tau}^{\tau_0}\frac{\mathbb E[(\Delta^\ell)^2]}{(\ell-1)^2}\leq 12K^2B_0^2\times \int_{\tau-2}^{\infty} \frac{\ud u}{u^2} = \frac{12K^2 B_0^2}{\tau-2}.
\label{eq:proof-tausharp-3}
\end{align}

Next, note that $\{\Delta^{\to\tau'}\}_{\tau'}$ is a martingale difference sequence. Therefore, $\{(\Delta^{\to\tau'})^4\}_{\tau'}$
is a sub-martingale. By Doob's martingale inequality we have 
\begin{align}
\Pr\left[\sup_{\tau\leq\tau'\leq\tau_0}\big|\Delta^{\to\tau'}\big|>\rho_0\right] 
&\leq \frac{\mathbb E[|\Delta^{\to\tau}|^2]}{\rho_0^2} \leq \frac{12K^2 B_0^2}{\rho_0^2(\tau-2)}.
\label{eq:proof-tausharp-3half}
\end{align}

We next focus on $\tilde\gamma^\tau_j$ for some resource type $j\in[M]$. 
Decompose the first term of $\varepsilon^\tau(j)=\sum_{i=1}^NA_{ij}\Delta^\tau(i)-\min\{\gamma^\tau_j,A_0B_0K\}\Delta^{\tau}$ as
$\sum_{i=1}^N A_{ij}\Delta^\tau(i) = \big[\sum_{i=1}^N A_{ij}v_ix^\tau_i - \sum_{i\in S^\tau}A_{ij}v_i\big] + 
[\sum_{i\in S^\tau} A_{ij}(n^\tau(i)-v_i)]$.
With $A_{ij}\leq A_0$ for all $i\in[N]$, $\sum_{i=1}^N v_ix^\tau_i\leq B_0\|x^\tau\|_1\leq B_0K$ and Eq.~(\ref{eq:proof-tausharp-3}), we have that
$\mathbb E[|\varepsilon^\tau(j)|^2] \leq 9(A_0^2B_0^2K^2 + 12A_0^2B_0^2K^2 + A_0^2B_0^2K^2\times 12B_0^2K^2)
\leq 2700 A_0^2B_0^4K^4.$
Subsequently, using the same analysis as in Eq.~(\ref{eq:proof-tausharp-3}), we have
\begin{equation}
\mathbb E[|\varepsilon^{\to\tau}(j)|^2] =\sum_{\ell,\ell'=\tau}^{\tau_0}\frac{\mathbb E[(\varepsilon^\ell(j))^2(\varepsilon^{\ell'}(j))^2]}{(\ell-1)^2(\ell'-1)^2}
\leq \frac{2700 A_0^2B_0^4K^4}{\tau-2},
\label{eq:proof-tausharp-4}
\end{equation}
where $\const.$ is a numerical constant omitted here for presentation purposes.
Note that $\{\varepsilon^{\tau'}(j)\}_{\tau'}$ is a martingale difference sequence. This implies $\{|\varepsilon^{\tau'}(j)|^4\}_{\tau'}$ is a sub-martingale difference sequence. By Doob's martingale inequality we have
\begin{align}
\Pr\left[\sup_{\tau\leq\tau'\leq\tau_0}\big|\varepsilon^{\to\tau'}(j)\big|>\rho_0\right]\leq \frac{\mathbb E[|\varepsilon^{\to\tau}(j)|^2]}{\rho_0^2}\leq \frac{2700 A_0^2B_0^4K^4}{\rho_0^2(\tau-2)}.
\label{eq:proof-tausharp-5}
\end{align}

Combining Eqs.~(\ref{eq:proof-tausharp-3half},\ref{eq:proof-tausharp-5}) and using the union bound, we have
\begin{align*}
\mathbb E[\tau^\sharp] &=2 + \sum_{\tau=3}^{\tau_0}\Pr[\tau^\sharp\geq\tau]  \leq \sum_{\tau=2}^{\tau_0}\Pr\left[\forall\tau'>\tau, |\Delta^{\to\tau'}|\leq\rho\wedge |\varepsilon^{\to\tau}(j)|\leq\rho,\forall j\in[M]\right]\\
&\leq 2 + \sum_{\tau=3}^{\tau_0}(M+1)\times \const.\times\frac{A_0^2B_0^4K^4}{\rho_0^2(\tau-2)} \leq \const.\times A_0^2B_0^4K^4M\ln T/\rho_0^2,
\end{align*}
proves the desired result. 
{Here, $\const.$ is a placeholder for a numerical, absolute constant independent of any problem parameters or constants involved in technical assumptions.}
$\square$
\end{proof}

\subsection{Stability analysis of $\Psi$}

The third building block of our proof to Theorem \ref{thm:constant-regret} is to analyze the maximum objective value of $\Psi(\gamma^\tau,s^\tau)$ for
constraint parameters $\gamma^\tau,s^\tau$ that deviates from $\gamma_0$ and $s^{\tau_0}=1+\sum_i v_ix_i^F$.
The following lemma analyzes such deviation in terms of $\varepsilon^{\to\tau}=\{\varepsilon^{\to\tau}(j)\}_{j=1}^M\in\mathbb R^M$ and $\Delta^{\to\tau}$
defined in Eq.~(\ref{eq:defn-vareps}).
\begin{lemma}
There exists a fixed vector $\eta_\gamma\in\mathbb R_+^M$ and scalars $\eta_s,\eta_h\geq 0$ depending only on $\mathcal M=(A_{ij},v_i)_{i,j=1}^{N,M}$, $K$, $\gamma_0$ such that the following holds: 
$$
\Psi(\gamma^\tau,s^\tau) \geq \Psi(\tilde\gamma^\tau,\tilde s^\tau) \geq \Psi(\gamma_0,s^{\tau_0}) - \big\langle \eta_\gamma, \varepsilon^{\to\tau}\big\rangle - \eta_s \Delta^{\to\tau}
- \eta_h(\|\varepsilon^{\to\tau}\|_2^2+|\Delta^{\to\tau}|^2).
$$
Furthermore, the scalar $\eta_h$ satisfies $\eta_h\leq \frac{8 M^2 KN B_0^3}{\chi_0^2}$.
\label{lem:psi-stability}
\end{lemma}
\begin{proof}{Proof of Lemma \ref{lem:psi-stability}.}
The first inequality trivially holds because any solution feasible to $\Psi(\tilde\gamma^\tau,\tilde s^\tau)$ is also feasible to $\Psi(\gamma^\tau,s^\tau)$,
since $\gamma^\tau\geq \tilde\gamma^\tau$ and $s^\tau\geq\tilde s^\tau$ almost surely.

By Assumptions (A3), the set of binding/active constraints are the same for all $\{(\gamma,s): \|\gamma-\gamma_0\|_{\infty}\leq\rho_0\wedge |s-s^{\tau_0}|\leq\rho_0\}$.
It is a standard result that this condition implies the differentiability of $\Psi(\cdot,\cdot)$ on $\{(\gamma,s): \|\gamma-\gamma_0\|_{\infty}\leq\rho_0\wedge |s-s^{\tau_0}|\leq\rho_0\}$ (see, e.g., \cite{gal2012advances,tercca2020envelope}).
The derivative of $\Psi(\cdot,\cdot)$ can also be calculated by using the celebrated \emph{envelope theorem} (\cite{milgrom2002envelope}, see also 
\cite[Theorem 1.F.1]{takayama1985mathematical}, \citep[Theorem 1]{tercca2020envelope}).
More specifically, let $x^*$ be the optimal solution to $\Psi(\gamma,\eta)$, $\lambda_j^*$ be the Lagrangian multiplier of the constraint $\sum_{i=1}^N A_{ij}v_ix_i \leq \gamma^\tau_j(1+\sum_{i=1}^N v_ix_i)$
and $\lambda_s^*$ be the Lagrangian multiplier of the constraint $1+\sum_{i=1}^N v_ix_i\leq s^\tau$, {both Lagrangian multipliers corresponding to the optimal solution $x^*$.} Then
$\eta_{\gamma,j} = {\partial_{\gamma_j}}\Psi(\gamma,s) = \lambda_j^*(1+\sum_iv_ix_i) \geq 0$ and 
$\eta_{s} = {\partial_s}\Psi(\gamma,s) = \lambda_s^* \geq 0$.
{Furthermore, Lemma \ref{lem:hessian} in the appendix shows that $\Psi$ is twice differentiable with bounded Hessian $\partial^2\Psi$ on $\{(\gamma,s): \|\gamma-\gamma_0\|_{\infty}\leq\rho_0\wedge |s-s^{\tau_0}|\leq\rho_0\}$.}
Let $\eta_h:=\max_{\gamma,s}\|\partial^2\Psi(\gamma,s)\|_{\mathrm{op}}\leq \frac{8 M^2 KN B_0^3}{\chi_0^2}<+\infty$ where $\gamma,s$ are in the closed neighborhood of $\{(\gamma,s): \|\gamma-\gamma_0\|_{\infty}\leq\rho_0\wedge |s-s^{\tau_0}|\leq\rho_0\}$.
Subsequently, using Taylor expansion we have that
$
|\Psi(\tilde\gamma^\tau,\tilde s^\tau) - (\Psi(\gamma_0,s^{\tau_0}) + \langle \eta_\gamma,\tilde\gamma^\tau-\gamma_0\rangle + \eta_s(\tilde s^\tau-s^{\tau_0}))|
\leq \eta_h [\|\tilde\gamma^\tau-\gamma_0\|_2^2+(\tilde s^\tau-s^{\tau_0})^2].
$
Plugging in the definitions of $\tilde\gamma^\tau$ and $\tilde s^\tau$ we complete the proof of Lemma \ref{lem:psi-stability}. $\square$
\end{proof}

\subsection{Completing the proof}

Recall the definition that $\beta^\tau=\sum_{i=1}^N r_in^\tau(i)$ is the total revenue collected during the epoch $\tau$.
The expected total revenue of the proposed policy can be lower bounded as
$\mathbb E[\sum_{t=1}^T r_{i_t}] \geq \mathbb E[\sum_{\tau=\tau^\sharp}^{\tau_0} \beta^\tau]$.
Therefore, it suffices to lower bound $\mathbb E[\sum_{\tau=\tau^\sharp}^{\tau_0}\beta^\tau]$.
Next, fix an arbitrary epoch $\tau\geq\tau^\sharp$. We have that
\begin{align}
\mathbb E[\beta^\tau|\gamma^\tau,s^\tau] &= \sum_{i=1}^N r_i\mathbb E[n^\tau(i)|\gamma^\tau,s^\tau] = \sum_{i=1}^N r_iv_i\mathbb E[x_i^\tau|\gamma^\tau,s^\tau]
=  \Psi(\gamma^\tau,s^\tau).
\label{eq:proof-complete-2}
\end{align}
Invoke Lemma \ref{lem:psi-stability} and note that $\mathbb E[\varepsilon^{\to\tau}]=0$, $\mathbb E[\Delta^{\to\tau}] = 0$. We have that 
\begin{align}
\mathbb E[\Psi(\gamma^\tau,s^\tau)] \geq \Psi(\gamma_0,s^{\tau_0}) - \eta_h\mathbb E[\|\varepsilon^{\to\tau}\|_2^2+|\Delta^{\to\tau}|^2].
\label{eq:proof-complete-3}
\end{align}
By Eqs.~(\ref{eq:proof-tausharp-3},\ref{eq:proof-tausharp-4}) in the proof of Lemma \ref{lem:tausharp-bound}, we have for every $\tau$ and $j\in[M]$ that
\begin{align}
\mathbb E\left[|\varepsilon^{\to\tau}(j)|^2\right] \leq \frac{2700 A_0^2B_0^4K^4}{\tau-2}, \;\;\;\; \mathbb E\left[|\Delta^{\to\tau}|^2\right] \leq \frac{12K^2 B_0^2}{\tau-2}.
\label{eq:proof-complete-31}
\end{align}
Combining Eqs.~(\ref{eq:proof-complete-3},\ref{eq:proof-complete-31}), we obtain
$\mathbb E[\Psi(\gamma^\tau,s^\tau)] \geq  \Psi(\gamma_0,s^{\tau_0}) - \const.\times \eta_h A_0^2 B_0^4 K^4/(\tau-2)$.
Additionally, because $s^{\tau_0}=1+\sum_i v_ix_i^F$, by Lemma \ref{lem:phi-psi} we have that $\Psi(\gamma_0,s^{\tau_0}) = s^{\tau_0}\Phi(\gamma_0)$.
Incorporating the bounds on $\mathbb E[\Psi(\gamma^\tau,s^\tau)]$, $s^{\tau_0}$ into Eqs.~(\ref{eq:proof-complete-2},\ref{eq:proof-complete-3}) and using the law of total expectation, we obtain
\begin{equation}
\mathbb E[\beta^\tau] \geq s^{\tau_0}\Phi(\gamma_0) -\const.\times \eta_h A_0^2 B_0^4 K^4/(\tau-2) .
\label{eq:proof-complete-5}
\end{equation}

Because $\tau^\sharp$ is a stopping time, the Doob's optional stopping theorem implies that
\begin{align}
\mathbb E\left[\sum_{\tau=\tau^\sharp}^{\tau_0}\beta_\tau\right] &\geq s^{\tau_0}\Phi(\gamma_0)\times (\tau_0-\mathbb E[\tau^\sharp])-\const.\times\sum_{\tau=1}^T\frac{\eta_h A_0^2B_0^4K^4}{\tau}\nonumber\\
&\geq T\Phi(\gamma_0) - s^{\tau_0}\Phi(\gamma_0)\times \const.\times A_0^4B_0^8K^8M\ln T/\rho_0^4-\const.\times\sum_{\tau=1}^T\frac{\eta_h A_0^2B_0^4K^4}{\tau}\label{eq:proof-complete-6}\\
&\geq T\Phi(\gamma_0) - s^{\tau_0}\Phi(\gamma_0)\times \const.\times A_0^4B_0^8K^8M\ln T/\rho_0^4 - \const.\times \eta_h A_0^2B_0^4K^4\ln T.\label{eq:proof-complete-7}
\end{align}
where Eq.~(\ref{eq:proof-complete-6}) holds by invoking Lemma \ref{lem:tausharp-bound} and noting that $s^{\tau_0}\tau_0 = T$ by definition;
Eq.~(\ref{eq:proof-complete-7}) holds because $\sum_{\tau=1}^T \tau^{-1}=O(\ln T)$ thanks to the Harmonic series.
The regret of the proposed re-solving policy is then at most 
$\const.\times s^{\tau_0}\Phi(\gamma_0)\times A_0^2B_0^4K^4M\ln T/\rho_0^2  + \const.\times \eta_h A_0^2B_0^4K^4\ln T 
\leq \const.\times A_0^2B_0^4K^4M\ln T/\rho_0^2 + \const.\times \frac{8 M^2 KN B_0^3}{\chi_0^2}\times A_0^2B_0^4K^4\ln T$,
which proves Theorem \ref{thm:constant-regret}.

\section{Numerical results}\label{sec:numerical}

We verify the effectiveness of the proposed resolving heuristics using numerical experiments.
{We compare the proposed Algorithm \ref{alg:resolve} (denoted as \textsf{Resolving} throughout this section) with the fluid approximation upper bound $\Phi(\gamma)$,
as well as two other obvious baseline methods: \textsf{Sampling-per-period} and \textsf{Sampling-per-epoch},
which involve solving the fluid problem only once and then sampling from it in an independent manner.
Details of these two baseline methods are given in Appendix \ref{appsec:numerical}.
 Details of experimental settings we implemented are also placed in Appendix \ref{appsec:numerical}.
}

\begin{figure}[t]
\centering
\includegraphics[width=0.32\textwidth]{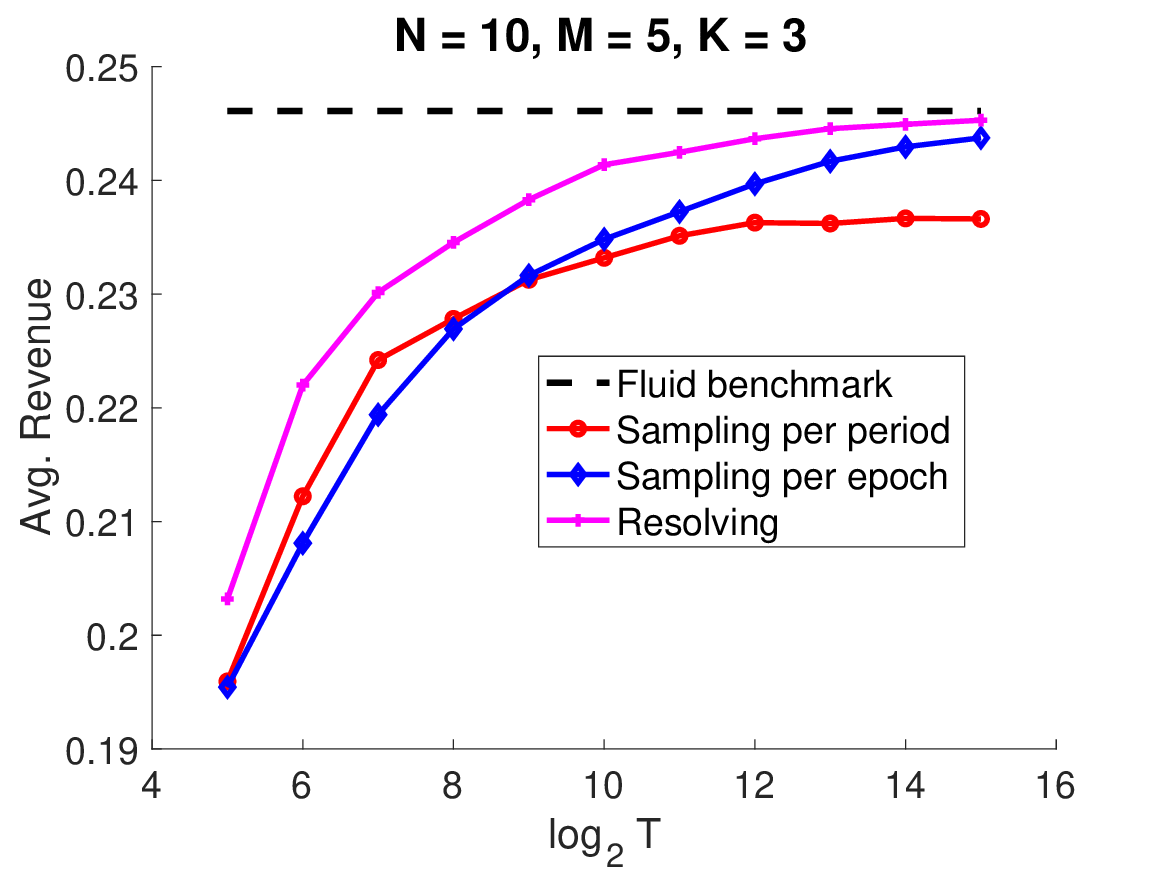}
\includegraphics[width=0.32\textwidth]{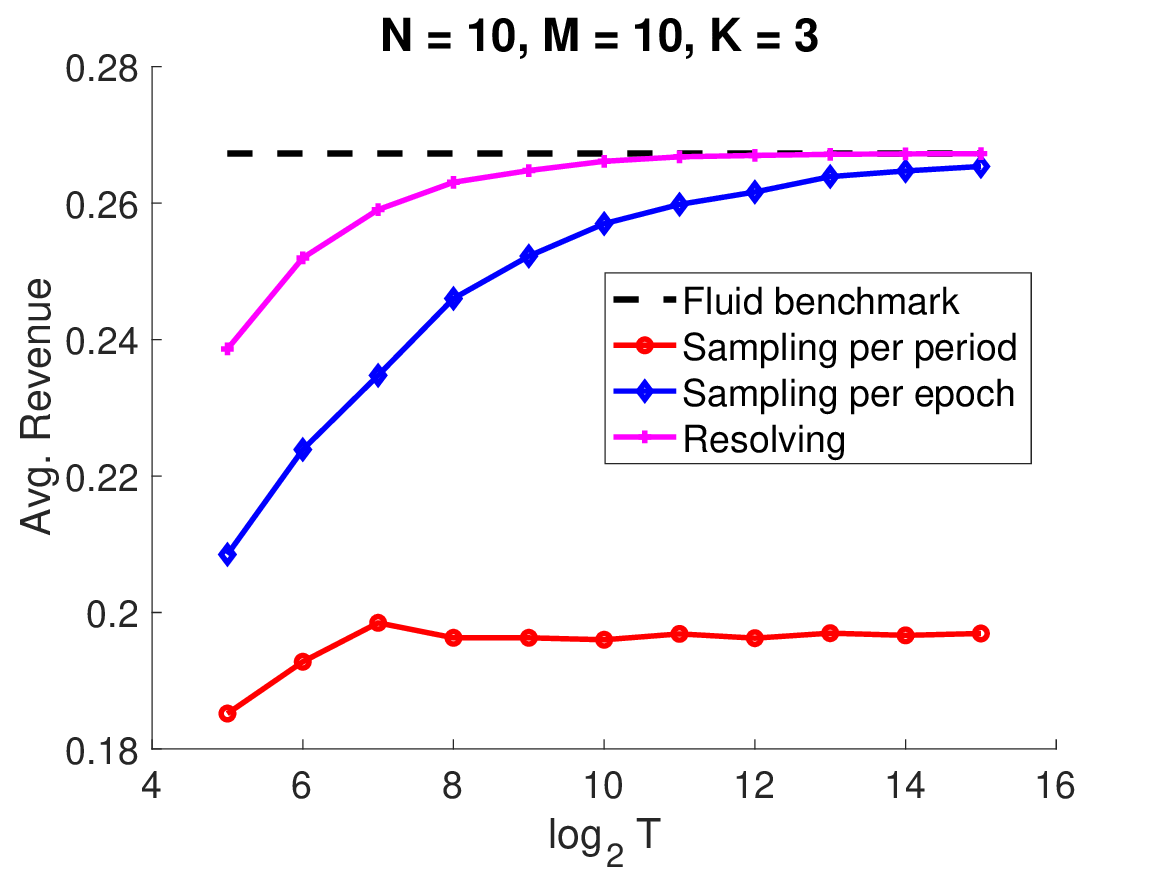}
\includegraphics[width=0.32\textwidth]{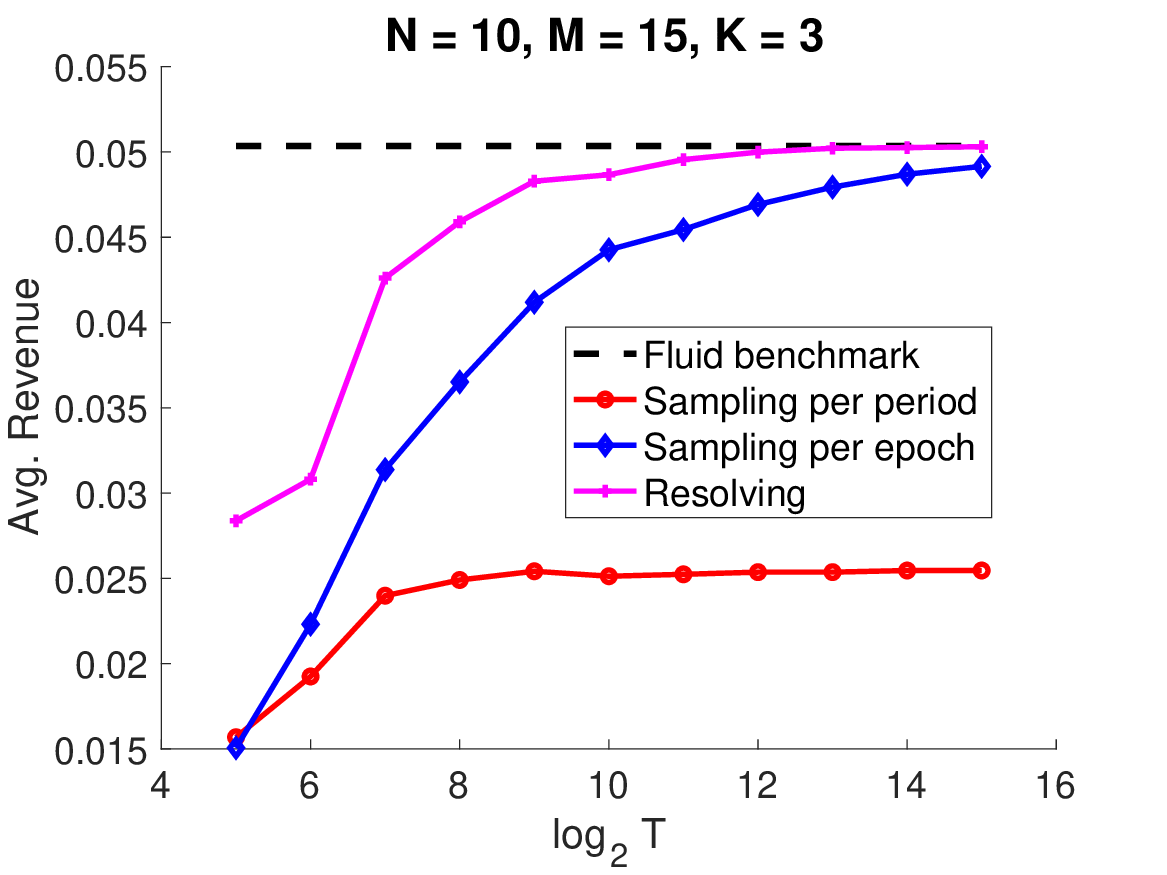}
\includegraphics[width=0.32\textwidth]{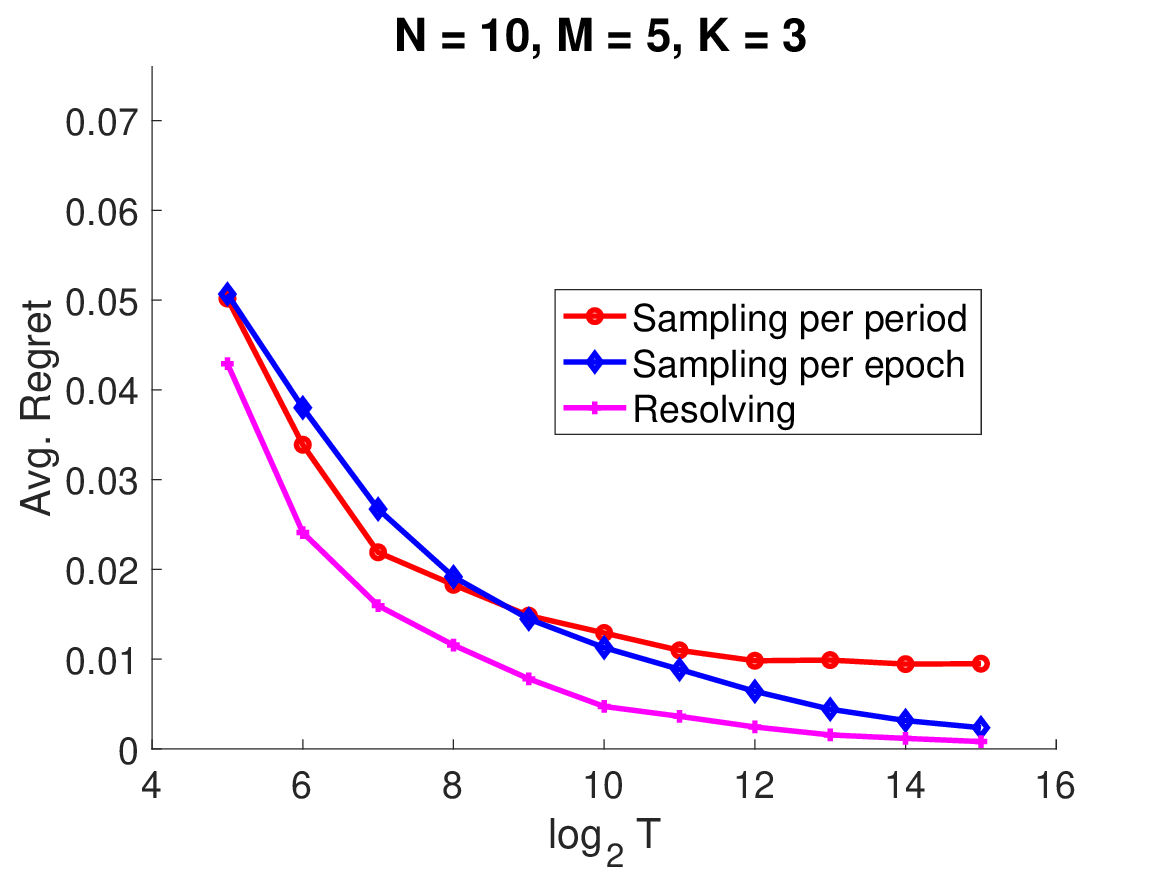}
\includegraphics[width=0.32\textwidth]{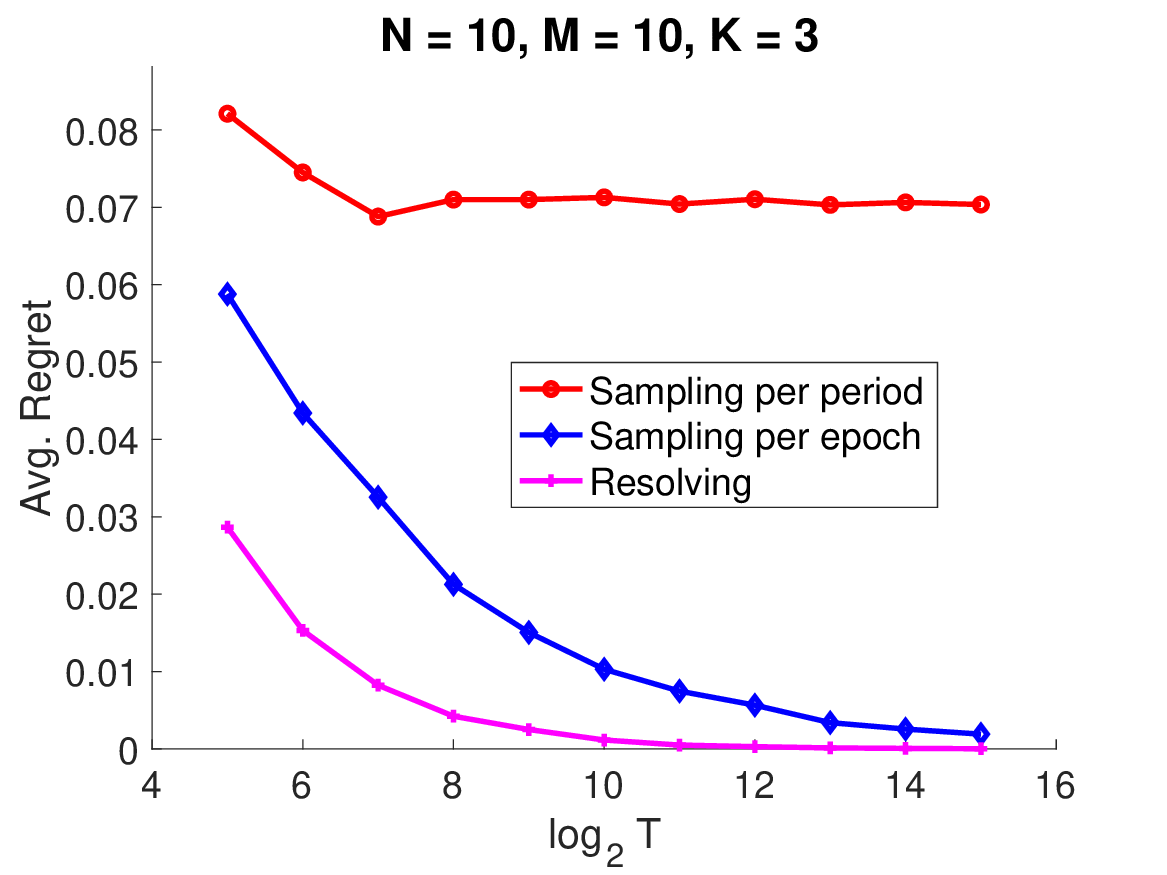}
\includegraphics[width=0.32\textwidth]{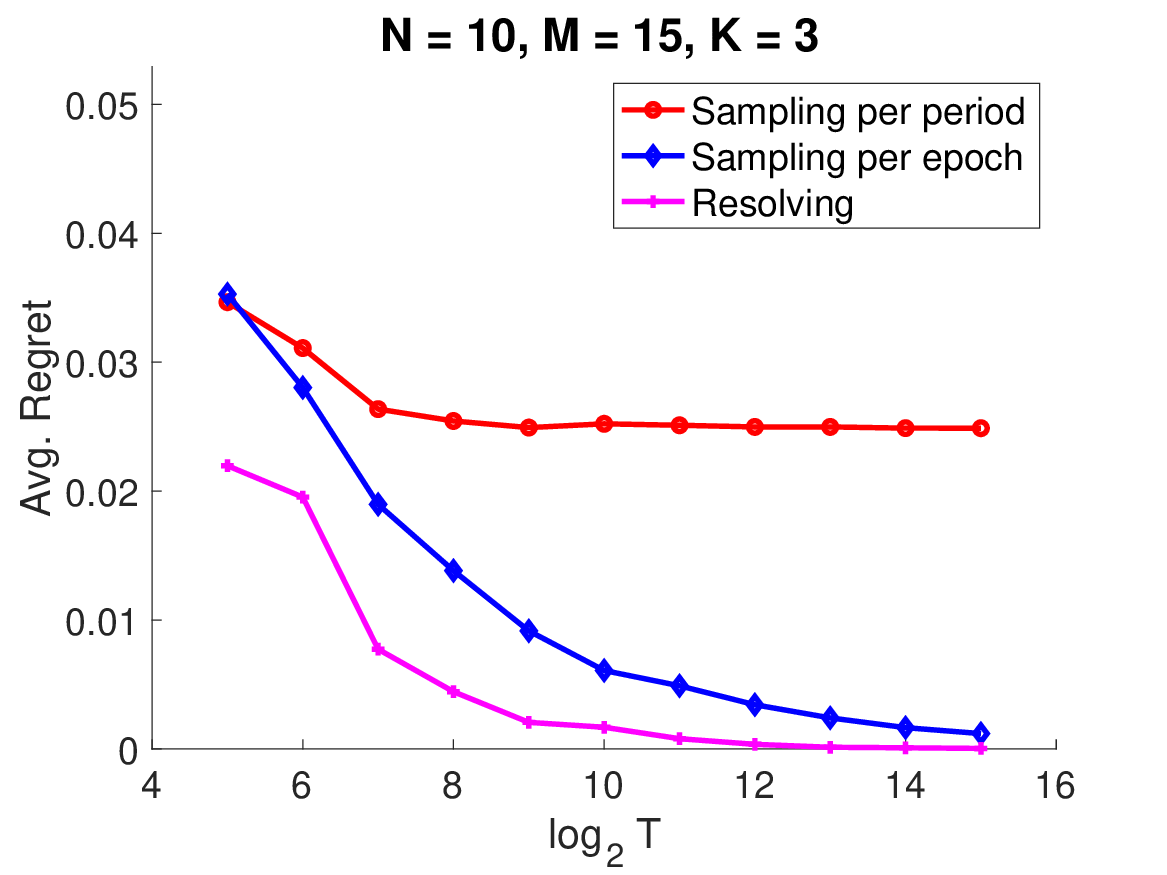}
\caption{Average revenue (top panel) and regret (bottom) performances for $N=10$ products and $K=3$ capacity constraint,
with the number of resources $M$ ranging from 5 to 15 and the number of time periods $T$ ranging from $2^5=32$ to $2^{15}=32768$.}
\label{fig:result_N10}
\end{figure}

In Figure \ref{fig:result_N10} we report the average revenue and regret performances of our proposed algorithm and the two mentioned baseline methods.
Each method is run for 500 independent trials and the mean revenue/regret is reported.
As we can see from Figure \ref{fig:result_N10}, the \textsf{Sampling-per-period} method has an inferior performance because of the non-linearity
of the $x\mapsto v^\top x/(1+v^\top x)$ function. In fact, the average revenue of the \textsf{Sampling-per-period} method 
does \emph{not} converge to the fluid approximation benchmark even when the time horizon $T$ is very large, indicating that it has linear regret.
On the other hand, the average revenue of both \textsf{Sampling-per-epoch} and \textsf{Resolving} converges to the fluid benchmark as $T$ increases,
indicating that both methods have sub-linear regret.
Furthermore, \textsc{Resolving} clearly outperforms \textsf{Sampling-per-epoch} to various degrees under virtually all different settings,
demonstrating the practical effectiveness of re-solving fluid optimization problems in the process of dynamic assortment optimization.

\begin{figure}[t]
\centering
\includegraphics[width=0.32\textwidth]{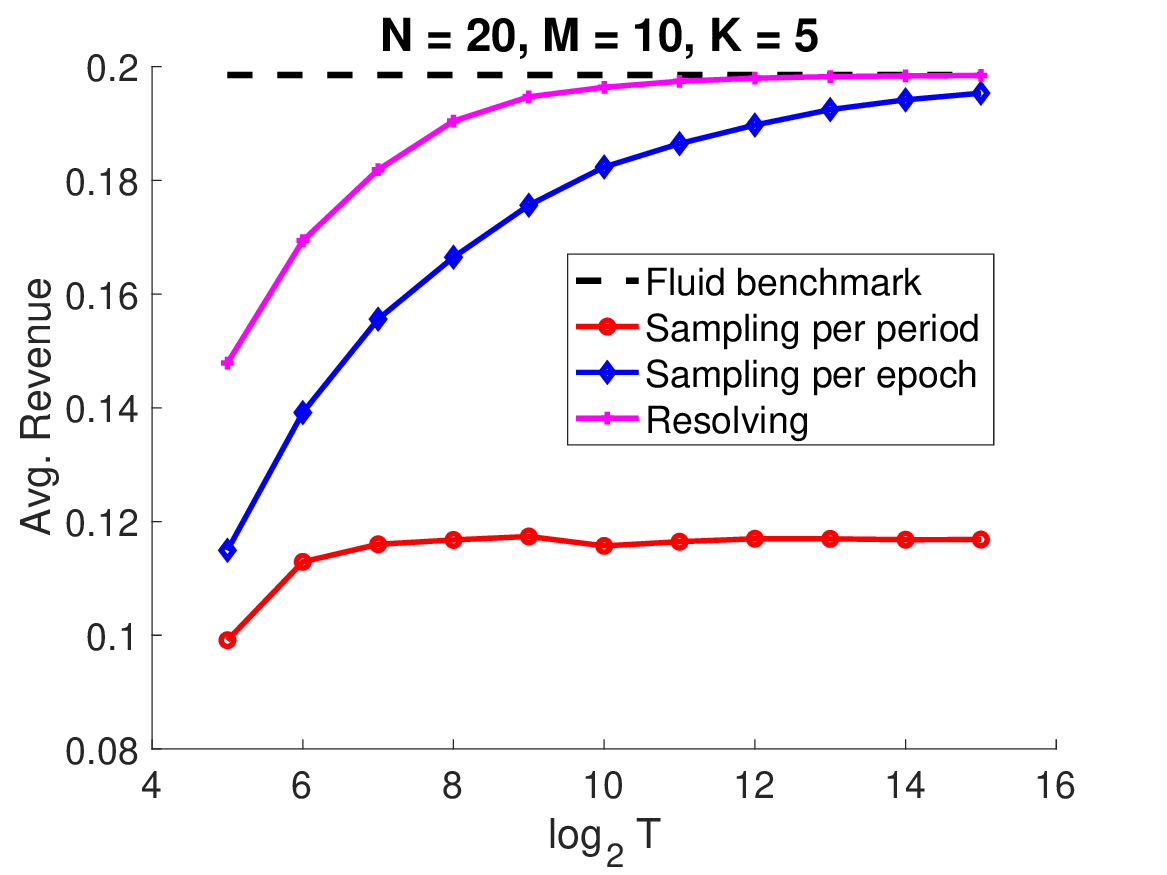}
\includegraphics[width=0.32\textwidth]{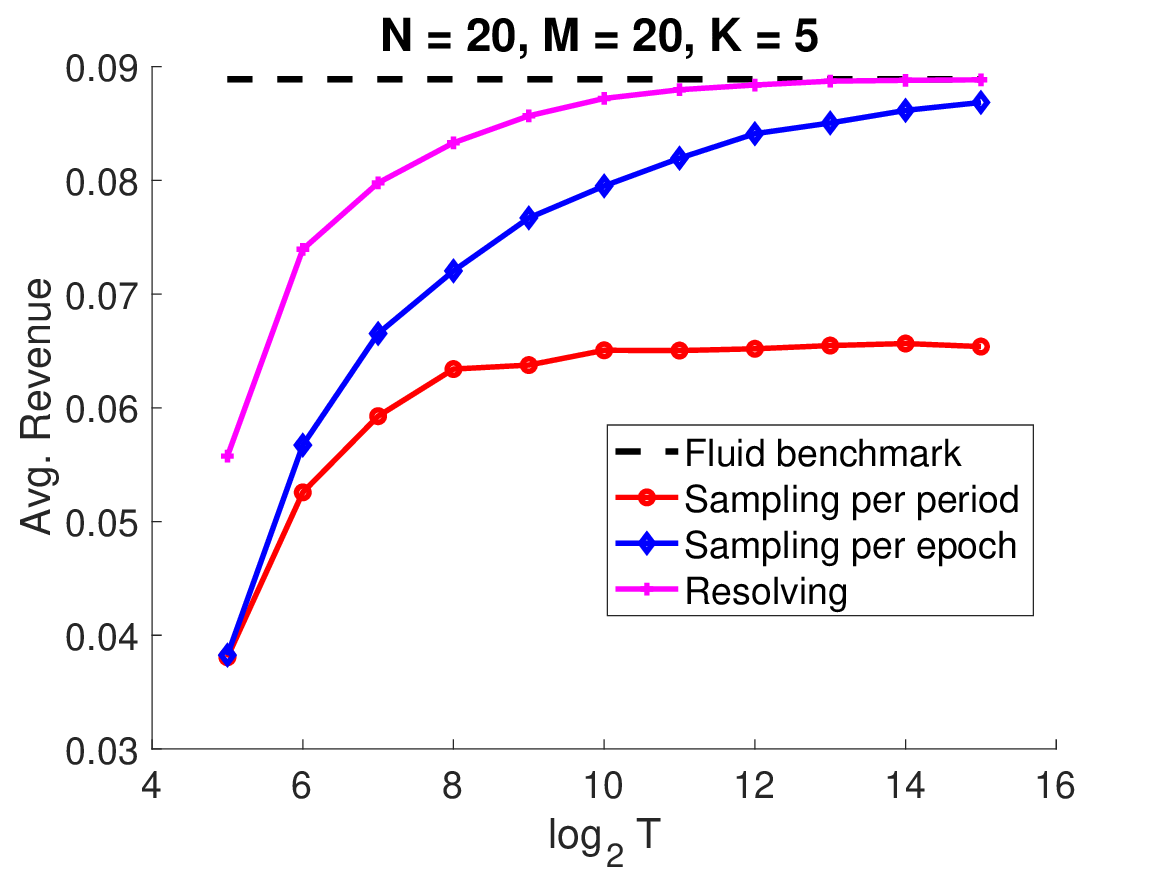}
\includegraphics[width=0.32\textwidth]{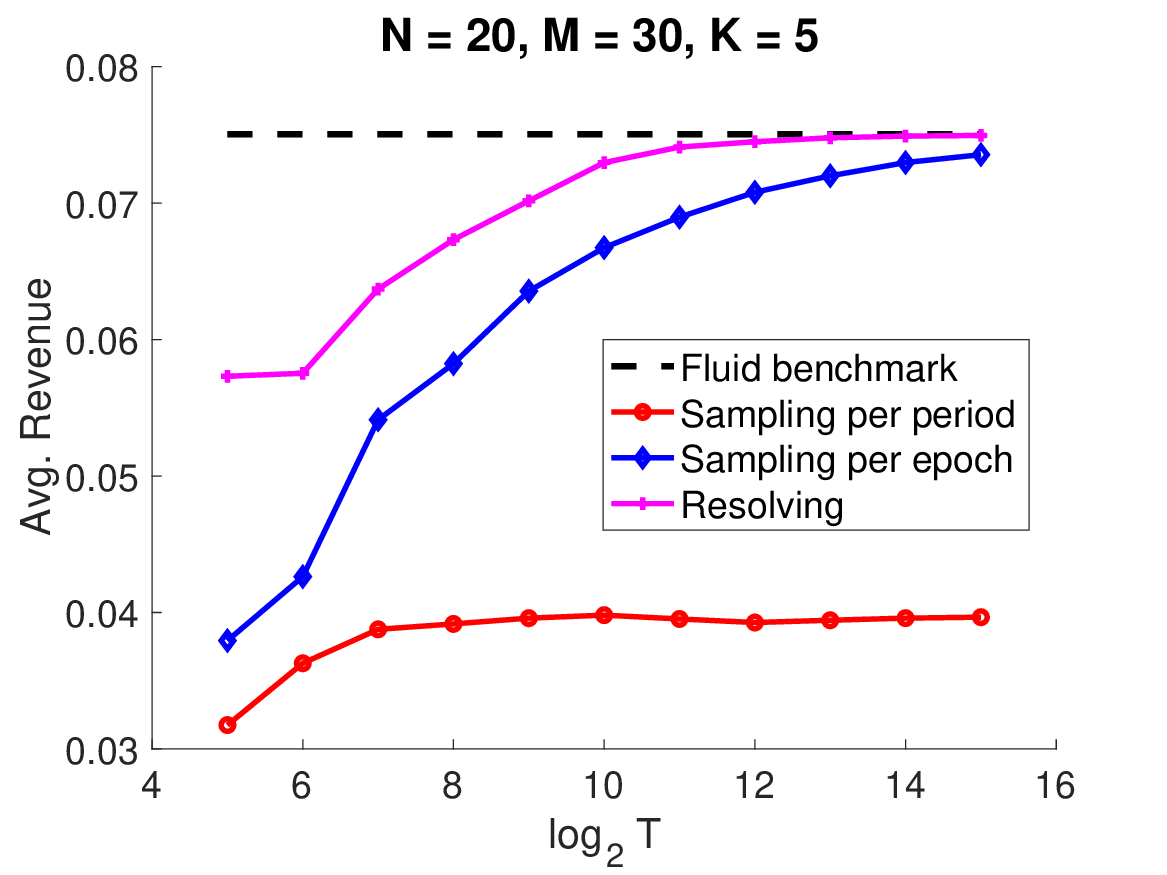}
\includegraphics[width=0.32\textwidth]{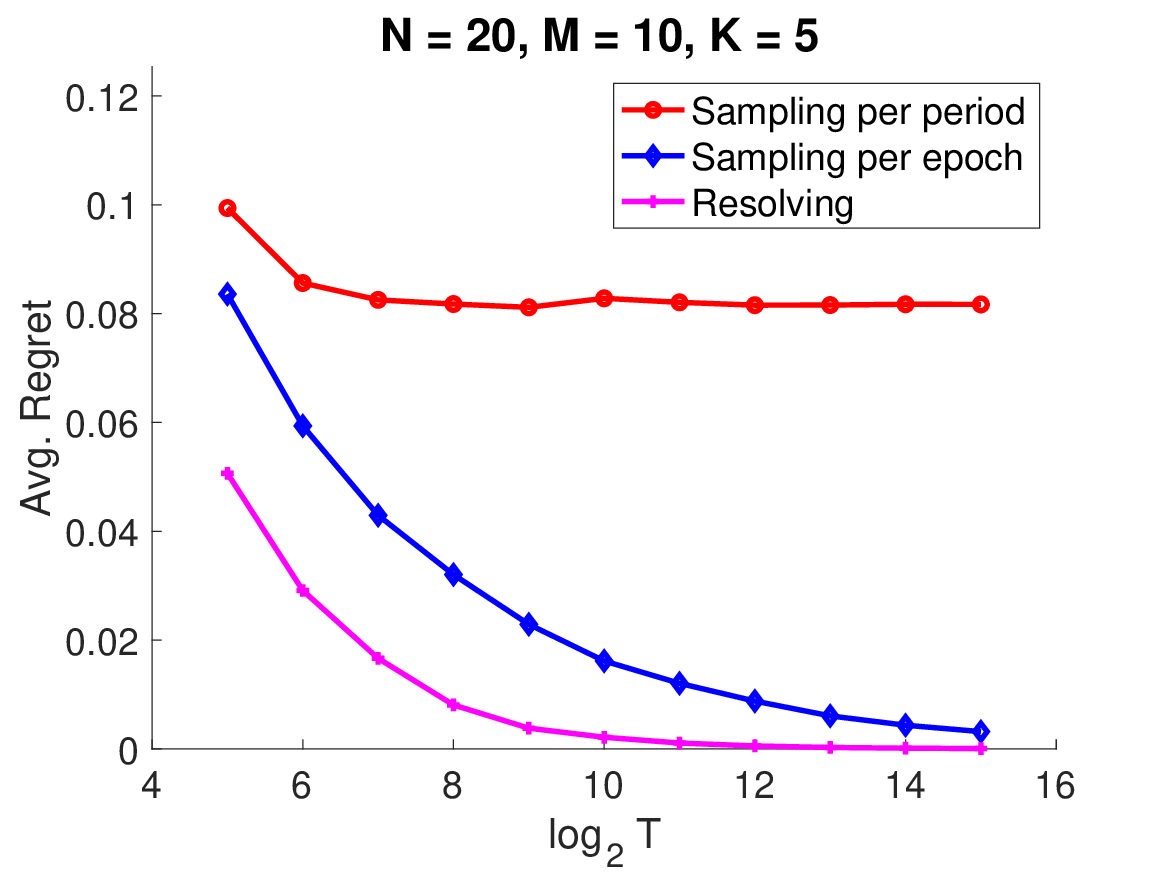}
\includegraphics[width=0.32\textwidth]{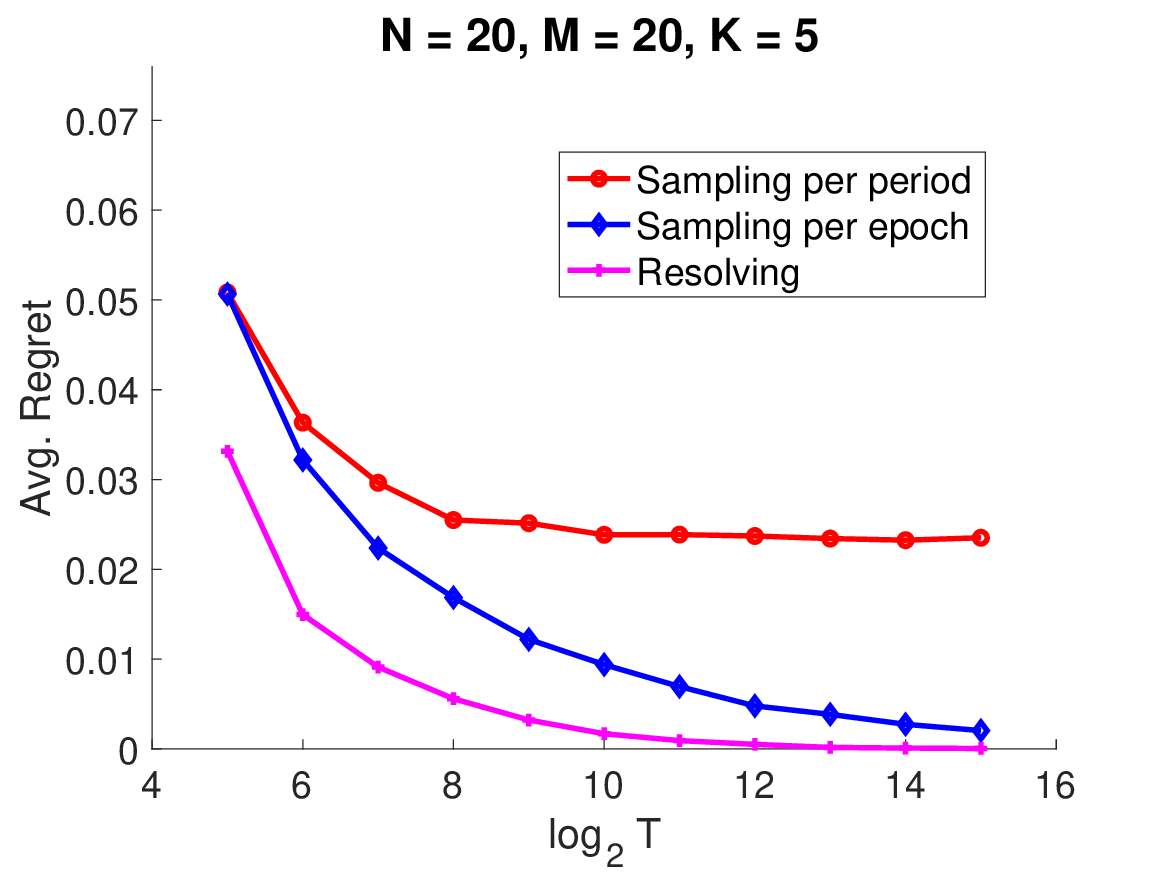}
\includegraphics[width=0.32\textwidth]{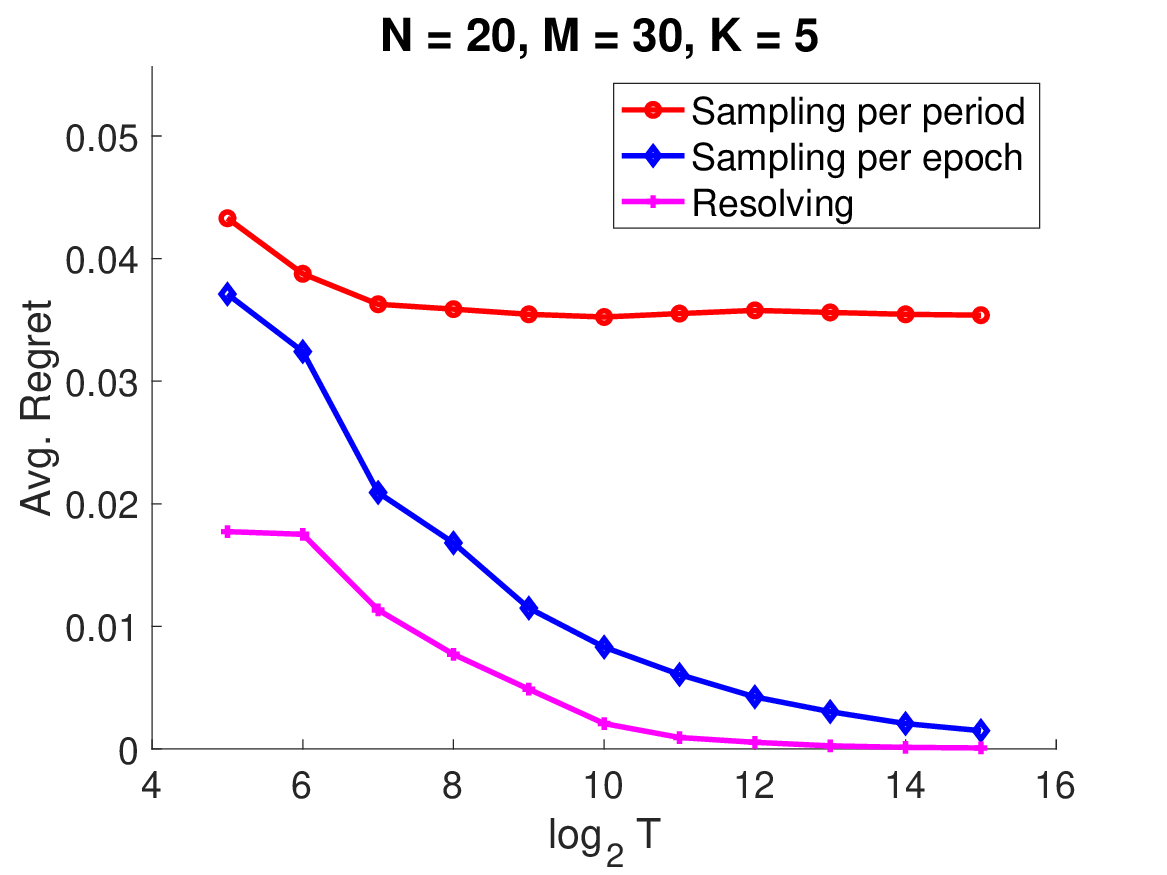}
\caption{Average revenue (top panel) and regret (bottom) performances for $N=20$ products and $K=5$ capacity constraint,
with the number of resources $M$ ranging from 10 to 30 and the number of time periods $T$ ranging from $2^5=32$ to $2^{15}=32768$.}
\label{fig:result_N20}
\end{figure}

In Figure \ref{fig:result_N20} we further report the average revenue and regret of the algorithms on slightly larger problem instances with $N=20$ products
and the number of resources $M$ ranging from 10 to 30. We have also increased the cardinality constraint $K$ from 3 to 5.
The results are similar to those reported in Figure \ref{fig:result_N10}, with even larger performance gaps between \textsf{Sampling-per-epoch} and \textsf{Resolving} on these larger problem instances.

Finally, we remark that both Figures \ref{fig:result_N10} and \ref{fig:result_N20} show that our proposed method has superior performance compared with baseline methods even for shorter time horizons where $T$ is small. This empirical observation complements our theoretical findings in Theorem \ref{thm:constant-regret}
which is focused primarily for longer time horizons.
{We also remark that both benchmarks are computationally cheaper than the resolving policy because they only need to compute the optimal solution to the fluid approximation once;
nevertheless, both benchmarks still require obtaining fresh samples of assortments every time period or epoch, invoking Algorithm \ref{alg:sampling-fluid} many times.
}

\section{Conclusions and Future Directions}
\label{sec:conclusion}

In this paper, we propose a new epoch-based re-solving technique for dynamic assortment optimization under knapsack inventory constraints. The proposed re-solving technique addresses several technical challenges arising from the non-linear objective due to the MNL choice model.
There are a few important future directions. For example, it would be interesting to extend the current problem into a learning problem where the choice model is unknown. The learning problem is closely related to multi-armed bandit (MAB) with knapsack constraints \citep{badanidiyuru2018bandits}. On the other hand, the special structure of MNL could make this problem much more complicated than the standard (MAB). In addition, it would be interesting to consider more complicated choice models beyond the standard MNL model, such as nested logit model \citep{Williams77,McFadden1980} or other choice models (see, e.g., \cite{Farias2013,Desir2015} and references therein).


\section*{Acknowledgement}
We thank the department editor, the senior editor and two anonymous referees for their helpful suggestions that greatly improved this paper.
Yuan Zhou was supported by the National Natural Science Foundation of China Grant 52494974.

\begin{appendices}
	
\vspace{.2cm}

{
\section{Proof of Lemma \ref{lem:fluid}}
\label{app:fluid}

We first define some notations.
Let $\mathcal S := \{S\subseteq[N]: |S|\leq K\}$ denote the set of all possible assortments of size at most $K$.
Let $\Delta^{\mathcal S} := \{q\in[0,1]^{|\mathcal S|}: \sum_{S\in\mathcal S}q_S=1\}$ denote the set of all probability distributions over $\mathcal S$.
For any $S\in\mathcal S$, define
$$
R(S) := \frac{\sum_{i\in S}r_iv_i}{1+\sum_{i\in S}v_i}, \;\;\;\;\;\;\nu_i(S) = \frac{\vct 1\{i\in S\}v_i}{1+\sum_{j\in S} v_j}
$$
as the expected revenue and the probability of choosing item $i$ when assortment $S$ is produced.
Let also $\nu(S) = (\nu_1(S),\cdots,\nu_N(S))\in\mathbb R^N$ be the concatenated expected purchase probability vector.

\paragraph{An alternative fluid approximation.} We present the following optimization problem, serving as an alternative fluid approximation:
\begin{align}
OPT_1 &:= \max_{q\in\Delta^{\mS}} \sum_{S\in\mS}q_S R(S) \;\;\;\;\;\;s.t.\;\; \sum_{S\in\mS}q_S A^\top\nu(S) \leq \gamma_0.
\label{eq:fluid-1}
\end{align}
Note that Eq.~(\ref{eq:fluid-1}) is an LP having $|\mS|$ variables and $N$ constraints. Because it has an enormously large number of variables, it is not suitable for practical optimization but the formulation serves as an important tool for mathematical analysis of other optimization problems.
An important result below shows that $OPT_1$ upper bounds the (normalized) expected revenue of any admissible policy.
\begin{lemma}
For any admissible policy $\pi$ over $T$ periods, it holds that $\mathbb E^\pi[\sum_{t=1}^T R(S_t)] \leq T\times OPT_1$.
\label{lem:fluid-1-valid}
\end{lemma}
\begin{proof}{Proof of Lemma \ref{lem:fluid-1-valid}.}
Let $S_1,\cdots,S_T$ be the assortments offered by policy $\pi$ over $T$ periods, which are correlated random variables forming a Markov chain.
For any $S\in\mS$, define $\hat q_S := \frac{1}{T}\sum_{t=1}^T\mathbb E^\pi[\vct 1\{S_t=S\}]$.
Clearly $\hat q_S\geq 0$ and $\sum_{S\in\mS}\hat q_S=1$ by definition.
Because $\pi$ is an admissible policy, $\sum_{t=1}^T A^\top\nu(S_t)\leq\gamma_0 T$ almost surely and therefore $\sum_{t=1}^T A^\top\mathbb E^\pi[\nu(S_t)]\leq\gamma_0 T$,
implying that $\hat q=(\hat q_S)_{S\in\mS}$ is feasible to Eq.~(\ref{eq:fluid-1}) because
\begin{align*}
\sum_{S\in\mathcal S}\hat q_S A^\top\nu(S) = \sum_{S\in\mathcal S}\frac{1}{T}\sum_{t=1}^T\mathbb E^\pi[\vct 1\{S_t=S\}] A^\top\nu(S) = \frac{1}{T}\sum_{t=1}^TA^\top\mathbb E^\pi[\nu(S_t)] \leq \gamma_0.
\end{align*}
Subsequently, $\mathbb E^\pi[\sum_{t=1}^T R(S_t)] = T\times \sum_{S\in\mS}\hat q_SR(S) \leq T\times OPT_1$, which is to be proved. $\square$
\end{proof}

{
\begin{remark}
Eq.~(\ref{eq:fluid-1}) is a general formulation for fluid approximation in sequential decision making with aggregated inventory constraints and unchanged action spaces, and Lemma \ref{lem:fluid-1-valid}
is somewhat folklore which can be found, perhaps in different forms, in many existing literature (e.g.~\cite{miao2021general}). 
On the other hand, the fact (which we are going to rigorously establish in the remainder of this section) that Eq.~(\ref{eq:fluid-1}) has the same optimal objective as Eq.~(\ref{eq:fluid}) is highly non-trivial and
is a consequence of strong duality of both formulations sharing the same dual function in this particular application of assortment optimization.
\end{remark}
}

It is also helpful to derive the Lagrangian dual function of Eq.~(\ref{eq:fluid-1}). Let $\lambda_j\geq 0$ be a Lagrangian multiplier associated with the constraint involving $\gamma_{0j}$ on the right-hand side of the constraints. Let also $\lambda = (\lambda_1,\cdots,\lambda_M)\in\mathbb R_+^M$ be the concatenated Lagrangian multiplier vector. The dual function of Eq.~(\ref{eq:fluid-1}) is then derived as
\begin{align}
g_1(\lambda) &:= \max_{q\in\Delta^{\mS}} \sum_{S\in\mS}R(S)q_S - \lambda^\top \left(\sum_{S\in\mS}q_SA^\top\nu(S)-\gamma_0\right)
= \max_{S\in\mS} R(S)-\lambda^\top A^\top \nu(S) + \lambda^\top\gamma_0\label{eq:proof-fluid-ap1}\\
&= \max_{S\in\mS} \frac{\sum_{i\in S}(r_i-\varphi_i)v_i}{1+\sum_{i\in S}v_i} + \lambda^\top\gamma_0,\label{eq:proof-fluid-ap2}
\end{align} 
where $\varphi := \lambda^\top A^\top\in\mathbb R^N$, and the second inequality in Eq.~(\ref{eq:proof-fluid-ap1}) holds because we are optimizing a linear function over a probability simplex.

\paragraph{Dual function of fluid approximation in Definition \ref{defn:fluid}.} Define $\beta=(\beta_1,\cdots,\beta_N)\in\mathbb R^N$ where $\beta_i=r_iv_i$, and $v=(v_1,\cdots,v_N)\in\mathbb R^N$. Define alse $B\in\mathbb R^{N\times M}$ as $B_{ij}=A_{ij}v_i$.
Let $\mathcal X := \{x\in[0,1]^N: \vct 1^\top x\leq K\}$. The fluid approximation problem in Definition \ref{defn:fluid} can then be written as
\begin{align}
OPT_2 := \max_{x\in\mathcal X} \frac{\beta^\top x}{1+v^\top x} \;\;\;\;\;\;s.t.\;\; \frac{B^\top x}{1+v^\top x}\leq\gamma_0.
\label{eq:fluid-2}
\end{align}

We next derive the dual problem to Eq.~(\ref{eq:fluid-2}). Let $\lambda\in\mathbb R_+^M$ be the Lagrangian vector. Then the dual function can be derived as
\begin{align}
g_2(\lambda) = \max_{x\in\mX}\frac{\beta^\top x}{1+v^\top x} - \lambda^\top\left(\frac{B^\top x}{1+v^\top x}-\gamma_0\right) = \max_{x\in\mX}\frac{(\beta- B\lambda)^\top x}{1+v^\top x} + \lambda^\top\gamma_0.
\label{eq:proof-fluid-ap3}
\end{align}

\paragraph{Fluid approximation via Charnes-Cooper transform.} Note that Eq.~(\ref{eq:fluid-2}) can be re-formulated as a linear fractional programming as
\begin{align}
OPT_2 = \max_{x\in\mX} \frac{\beta^\top x}{1+v^\top x} \;\;\;\;\;\;s.t.\;\; (B^\top-\gamma_0v^\top)x\leq\gamma_0,
\label{eq:fluid-2half}
\end{align}
because $1+v^\top x\geq 1$ for all $x\in\mX$ and therefore an ($1+v^\top x$) term can be multiplied on both sides of inequality constraints in Eq.~(\ref{eq:fluid-2}).
Using the celebrated Charnes-Cooper transform of linear fractional programming (see, e.g.~\citep{charnes1962programming}, and also Appendix \ref{sec:cp} this paper for a brief summary)
$z=x/(1+v^\top x)$ and $t=1/(1+v^\top x)$, Eq.~(\ref{eq:fluid-2half}) is equivalent to
\begin{align*}
OPT_2 = \max_{z,t}\beta^\top z\;\;\;\;\;\;s.t.\;\; (B^\top-\gamma_0v^\top) z\leq\gamma_0 t, \;\; v^\top z+t=1, \;\; t\geq 0, \;\; z/t\in\mX,
\end{align*}
which is further equivalent to
\begin{align}
OPT_2 = \max_{z,t\in\mZ} \beta^\top z\;\;\;\;\;\;s.t.\;\; (B^\top-\gamma_0v^\top)z\leq\gamma_0 t,
\label{eq:fluid-3}
\end{align}
where $\mZ = \{(z,t): v^\top z+t=1,z\geq 0,t\geq 0,z\leq t,1^\top z\leq Kt\}$.
Note that Eq.~(\ref{eq:fluid-3}) is an LP over $z,t$, and the domain $\mZ$ is a polytope in $\mathbb R^{N+1}$.

It is also instructive to derive the Lagrangian dual function of Eq.~(\ref{eq:fluid-3}).
For any Lagrangian multiplier vector $\lambda\in\mathbb R_+^M$, the dual function is derived as
\begin{align}
g_3(\lambda) := \max_{z,t\in\mZ}\beta^\top z-\lambda^\top ((B^\top-\gamma_0 v^\top)z-\gamma_0 t) = \max_{z,t\in\mZ}(\beta-B\lambda)^\top z + \lambda^\top\gamma_0,
\label{eq:proof-fluid-ap4}
\end{align}
where the second equality holds because $t=1-v^\top z$ for any $(z,t)\in\mZ$.

We next present the central lemma to our proof, showing that the three dual functions $g_1(\cdot),g_2(\cdot),g_3(\cdot)$ are equivalent.
\begin{lemma}
For any $\lambda\in\mathbb R_+^M$, $g_1(\lambda)=g_2(\lambda)=g_3(\lambda)$.
\label{lem:dual-equivalence}
\end{lemma}
\begin{proof}{Proof of Lemma \ref{lem:dual-equivalence}.}
Fix arbitrary $\lambda\in\mathbb R_+^M$. We first focus on dual functions $g_1$ and $g_2$.
Consider arbitrary $S^*\in\arg\max_{|S|\leq K}\frac{\sum_{i\in S}(r_i-\varphi_i)v_i}{1+\sum_{i\in S}v_i}$.
Define $x_i := \vct 1\{i\in S^*\}$. It is easy to verify that $x\in\mX$. Furthermore, $\frac{\sum_{i\in S}(r_i-\varphi_i)v_i}{1+\sum_{i\in S}v_i} = \frac{(\beta-B\lambda)^\top x}{1+v^\top x}$. This shows that $g_1(\lambda)\leq g_2(\lambda)$.

To show the other direction, note that there exists $x^*\in\{0,1\}^N\cap\mX$ such that $\frac{(\beta-B\lambda)^\top x^*}{1+v^\top x^*} = g(\lambda)-\gamma_0^\top\lambda$.
To see the existence of such $x^*$, consider arbitrary $x^\sharp \in\arg\max_{x\in\mX}\frac{(\beta-B\lambda)^\top x}{1+v^\top x} $.
By definition, $\frac{(\beta-B\lambda)^\top x^\sharp}{1+v^\top x^\sharp} = g^*=:g(\lambda)-\gamma_0^\top\lambda$, which implies 
\begin{align}
(\beta-B\lambda-g^*v)^\top x^\sharp \geq g^*.
\label{eq:proof-dual-1}
\end{align}
For every $i\in[N]$ let $\varsigma_i := (\beta-B\lambda -g^*v)_i$. Now define $x_i^*=1$ if $\varsigma_i\geq 0$ and $\varsigma_i$ is within the top-$K$ of $\{\varsigma_i\}_{i=1}^N$ (ties broken arbitrarily),
and $x_i^*=0$ otherwise. It is easy to see that such defined $x^*\in\mX$ achieves the maximum possible value on the left-hand side of Eq.~(\ref{eq:proof-dual-1}).
Therefore, $(\beta-B\lambda-g^*v)^\top x^* \geq g^*$, which implies $\frac{(\beta-B\lambda)^\top x^*}{1+v^\top x^*} \geq g^*$.
Because $g(\lambda)$ maximizes over all $x\in\mX$, this actually means $\frac{(\beta-B\lambda)^\top x^*}{1+v^\top x^*} = g^*$.

With such an $x^*\in\{0,1\}^N\cap\mX$ that $\frac{(\beta-B\lambda)^\top x^*}{1+v^\top x^*} = g(\lambda)-\gamma_0^\top\lambda$,
define $S:= \{i\in[N]: x_i^*=1\}$. Clearly $S\in\mS$ and $\frac{\sum_{i\in S}(r_i-\varphi_i)v_i}{1+\sum_{i\in v_i}} = \frac{(\beta-B\lambda)^\top x^*}{1+v^\top x^*}$.
This shows that $g_2(\lambda) \leq g_1(\lambda)$. Combined with the argument in the previous paragraph we proved $g_1(\lambda)=g_2(\lambda)$.

We next focus on the dual functions $g_2$ and $g_3$. For any $x^*\in\arg\max_{x\in\mX}\frac{(\beta-B\lambda)^\top x}{1+v^\top x}$,
define $z := x^* / (1+v^\top x^*)$ and $t:= 1/(1+v^\top x^*)$. It is easy to see that $z,t\in\mZ$ and furthermore $\frac{(\beta-B\lambda)^\top x^*}{1+v^\top x^*} = (\beta-B\lambda)^\top z$. This proves $g_2(\lambda)\leq g_3(\lambda)$. For the other direction, consider arbitrary $z^*,t^*\in\arg\max_{(z,t)\in\mZ}(\beta-B\lambda)^\top z$.
Define $x := z^*/t^*$. It is easy to verify that $x\in\mX$ and furthermore $(\beta-B\lambda)^\top z^* =\frac{(\beta-B\lambda)^\top (z^*/t^*)}{1/t^*}=\frac{(\beta-B\lambda)^\top(z^*/t^*)}{(v^\top z^*+t^*)/t^*} = \frac{(\beta-B\lambda)^\top(z^*/t^*)}{1+v^\top z^*/t^*} = \frac{(\beta-B\lambda)^\top x^*}{1+v^\top x^*}$.
This proves that $g_3(\lambda)\leq g_2(\lambda)$. Combining both directions we proved that $g_2(\lambda)=g_3(\lambda)$. $\square$
\end{proof}

With Lemma \ref{lem:dual-equivalence}, we know that all three dual problems have the same optimal value $G^* := \max_{\lambda\geq 0}g_1(\lambda) = \max_{\lambda\geq 0}g_2(\lambda) = \max_{\lambda\geq 0}g_3(\lambda)$.
Because $g_1$ and $g_3$ are dual functions of LPs, strong duality holds and therefore $OPT_1=G^*=OPT_2$. 
Combining this with Lemma \ref{lem:fluid-1-valid} we complete the proof of Lemma \ref{lem:fluid}.

\vspace{.2cm}
}

\section{Proofs of some technical lemmas}
\label{app:proof}

\vspace{.2cm}

\begin{proof}{Proof of Lemma \ref{lem:phi-psi}.}
Because $x^F$ is feasible and optimal to $\Phi(\gamma)$, it holds that $\|x^F\|_1\leq K$ and $A^\top\nu(x^F)\leq\gamma$.
Because $s$ is defined as $s=1+\sum_{i=1}^Nv_ix_i^F$, we know that $x^F$ is feasible to $\Psi(\gamma,s)$ too.
Furthermore, $\sum_{i=1}^Nr_iv_ix_i^F = R(x^F)\times (1+\sum_{i=1}^Nv_ix_i^F) = R(x^F)s$. Therefore, $\Psi(\gamma,s)\geq \sum_{i=1}^N r_iv_ix_i^F = R(x^F)s = \Phi(\gamma)s$.

Next, let $z^F$ the optimal solution to $\Psi(\gamma,s)$, where $s=1+\sum_{i=1}^Nv_ix_i^F$. We then have $\Psi(\gamma,s)=\sum_{i=1}^N r_iv_iz_i^F = R(z^F)\times[1+\sum_{i=1}^Nv_iz_i^F] \leq R(z^F)s$,
where the last inequality holds because $1+\sum_{i=1}^Nv_iz_i^F\leq s$ since $z^F$ is feasible to $\Psi(\gamma,s)$.
On the other hand, because $z^F$ is feasible to $\Phi(\gamma)$, from optimality we have that $R(z^F)\leq \Phi(\gamma) = R(x^F)$. Subsequently,
$\Psi(\gamma,s) \leq R(z^F)s \leq \Phi(\gamma) s$.

Combining both cases, we conclude that $\Psi(\gamma,s)=\Phi(\gamma)s$. $\square$
\end{proof}

\begin{proof}{Proof of Lemma \ref{lem:fluid-solver}.}
It is clear that $\hat x^F$ constructed in Algorithm \ref{alg:fluid} is also feasible because of the constraints in Eq.~(\ref{eq:lambda-fluid}).
In addition, it is easy to verify that if $(1-\lambda)[\sum_i\hat x^F_\lambda]_i\geq\lambda$ then $\Phi(\gamma)\geq\lambda$, and vice versa.
Hence, the bisection search procedure in Algorithm \ref{alg:fluid} will never exclude the optimal $R(\gamma)$ from $[\lambda_L,\lambda_U]$,
and furthermore at the end of the algorithm, $(1-\lambda_L)[\sum_i v_i\hat x_i^F] \geq \lambda_L$, which implies $R(\hat x^F)\geq\lambda_L = \lambda_U-\epsilon\geq R(\gamma)-\epsilon$.
Finally, Algorithm \ref{alg:fluid} solves $O(\log(1/\epsilon))$ LPs because this is the total number of iterations until $\lambda_U-\lambda_L\leq\epsilon$. $\square$
\end{proof}

\begin{proof}{Proof of Proposition \ref{prop:sampling-fluid}.}
The fact that $\|z\|_1\leq K$ almost surely is obvious from the observation that $\mat P_u$ is a permutation matrix, and hence the first $K$ rows of $\mat P_u$
consists of exactly $K$ ones.
To prove $\mathbb E[z_i]=x_i$, note that $z_i\in\{0,1\}$ and $\Pr[z_i=1] = \sum_{\ell=1}^L\sum_{k=1}^K\Pr[u=\ell]\mat P_{\ell,ki} = \sum_{\ell=1}^L\sum_{k=1}^K\alpha_\ell\mat P_{\ell,ki} = \sum_{k=1}^K\mat M_{ki} = x_i$.
This completes the proof. $\square$
\end{proof}

\section{Second-order stability of linear programming}

\begin{lemma}
	Under Assumptions (A1) to (A3), for any $\gamma,s$ satisfying $\|\gamma-\gamma_0\|_{\infty}\leq\rho_0$ and $|s-s^{\tau_0}|\leq\rho_0$, $\Psi$ is differentiable in $(\gamma,s)$ and furthermore $
	\|\partial^2 \Psi(\gamma,s)\|_{\mathrm{op}}\leq \frac{8 M^2 KN B_0^3}{\chi_0^2}.$
	\label{lem:hessian}
\end{lemma}

\begin{proof}{Proof.}
	Recall that $\Psi$ can be written as the following linear programming problem:
	\begin{align*}
		\Psi(\gamma,s) = & \max_{x\in [0,1]^N} \sum_{i=1}^N r_i v_i  x_i,s.t.\sum_{i=1}^Nx_i \leq K,  1+\sum_{i=1}^Nv_ix_i\leq s, \sum_{i=1}^N A_{ji}v_ix_i\leq \gamma_j\left[1+\sum_{k=1}^Nv_kx_k\right],\;\forall j\in[M].
	\end{align*}
	We convert the constraints of $\Psi$ into the standard form $\tilde{A} x \le \tilde b$, where $\tilde{A} \in \R^{(M+2+ 2N)\times N}$ denotes the coefficient matrix in the left hand side of constraints, and $\tilb \in \R^{M+2+ 2N}$ denotes the vector in the right hand side of constraints. The first row of matrix $\tilde{A}$ is an all-one vector. For row index $i$  from 2 to $M + 1$, each entry in $\tilde{A}$ is $(A_{(i-1), j} - \gamma_{i - 1} ) v_{j} $, for column index $j = 1,...,N$. The $(M+2)$-th row of $\tilde{A}$ is the same as $v$. The last $2N\times N$ submatrix of $\tilde{A}$ is $[\mathbb{I}_N, -\mathbb{I}_N]^\top$.  The vector $\tilb$ is $[K, \gamma, s-1, [1,...,1]_N, [0,...,0]_N]^\top$.
	We use $\tlam^* \in \R^{M+2+ 2N}$ to denote the optimal dual variables of constraints $\tilde{A} x \le \tilde b$ for $\Psi$.  Notice that both $\tlam^*$ and $x^*$ depend on $\gamma$ and $s$. 
	
	We first consider the first order derivative of $\Psi(\gamma, s)$ using the envelope theorem. We notice that $\gamma_j$ only appears in the $(j+1)$-th row in constraints, and $s$ only appears in the last row of constraints. Thus, according to the envelope theorem, the partial derivatives of \(\Psi\) with respect to the parameters \(\gamma\) and \(s\) are given by:
	$ \frac{\partial \Psi}{\partial \gamma_j} = \tlam_{j+1}^* \left( \frac{\partial \tilde b}{\partial \gamma_j} - \frac{\partial \tilde{A} }{\partial \gamma_j} x^* \right)_{(j+1)}, $ and 
	$ \frac{\partial \Psi}{\partial s} = \tlam_{M+2}^* \left( \frac{\partial\tilde  b}{\partial s} - \frac{\partial \tilde{A}}{\partial s} x^* \right)_{(M+2)}. $
	
	By the value of $\tilde{A}$ and $\tilde b$, we have that $  \left( \frac{\partial \tilde b}{\partial \gamma_j} - \frac{\partial \tilde{A} }{\partial \gamma_j} x^* \right)_{(j+1)} = 1 + \sum_{i} r_i v_i x_i$, and $ \left( \frac{\partial  \tilde b}{\partial s} - \frac{\partial \tilde{A}}{\partial s} x^* \right)_{(M+2)} = 1$. Thus, we obtain that $
		\frac{\partial \Psi}{\partial \gamma_j} =  \tlam_{j+1}^*(1 +\sum_{i} r_i v_i x_i), \text{for } j \in [1,....,M], \text{and } ~~\frac{\partial \Psi}{\partial s} = \tlam_{M+2}^* $.

	Next, we consider the second derivative of $\Psi(\gamma, s)$. We denote the basis of the binding constraints at $x^*$ by $\tB^*(\gamma, s) \in \R^{N\times N}$ and denote its corresponding right hand side by $\tilb_B \in \R^{N}$. Thus, we have that $\tB^*(\gamma, s) x^* = \tilb_B$. We further denote the optimal dual variables for the binding constraints by $\tlam_B \in \R^N$. By the complementary slackness, the optimal dual variables for nonbinding constraints are zero, so we have that $\tB^*(\gamma, s)^\top \tlam_B = v \odot r$, where $\odot$ represents the entrywise product of two vectors.
	
	The major part of anlayzing the hessian matrix $\partial^2 \Psi(\gamma,s)$ is to examine how $\tlam^*$ and $x^*$ depend on $(\gamma, s)$. By Assumption (A3), for any $(\gamma,s) \in \{(\gamma,s): \|\gamma - \gamma_0\| \le \rho_0 \wedge |s-s^{\tau_0}|\leq\rho_0  \}$,  the minimum singluar value of $\tB^*$ is greater than or equal to $\chi_0>0$. It implies that there exsits a unique solution $x^*$ for $\Psi(\gamma, s)$, and that the basis $\tB^*$ remains unchanged when $(\gamma, s)$ is within the set of $\{(\gamma,s): \|\gamma - \gamma_0\| \le \rho_0 \wedge |s-s^{\tau_0}|\leq\rho_0  \}$. Thus, for the non-binding constraints regarding to resource $j$, for any $k = 1,...,M$, we have that $ \frac{\partial \tlam^*_{j+1}}{ \partial \gamma_k }  = 0, ~~ \frac{\partial \tlam^*_{j+1}}{ \partial s }  = 0.$
	Similarly, for the non-binding constraints regarding to resource $j$, we have that 
	$ \frac{\partial x^*}{ \partial \gamma_j }  = 0.$
	
	Thus, when analyzing the upper bound for the operator norm of the hessian matrix, $\|\partial^2 \Psi(\gamma,s)\|_{\mathrm{op}}$, it suffices to consider the binding constraints. Since $\tB^*(\gamma, s)^\top \tlam_B = v$ and the minimum singular value of $\tB^*$ is no smaller than $\chi_0$,  we have that $\|v\| = \| \tB^*\tlam_B \| \ge\chi_0  \|\tlam_B\|$. Since $\|v\| \le \sqrt{N}B_0$, we have that $\|\tlam_B\| \le \frac{ \sqrt{N}B_0}{\chi_0}$.
	By taking the derivative of $(\gamma, s)$ on both sides of $\tB^*(\gamma, s)^\top \tlam_B = v\odot r$ , we obtain that $\frac{\partial \tB^*}{ \partial (\gamma, s) } \tlam_B +  \tB^*  \frac{\partial \tlam_B}{ \partial (\gamma, s) } = 0.$ Therefore, we have that $
		\left\| \tB^*  \frac{\partial \tlam_B}{ \partial (\gamma, s) } \right\| = \left\| \frac{\partial \tB^*}{ \partial (\gamma, s) } \tlam_B \right\|.$ 
	
	Since  $\tB$ is a submatrix of $\tilde{A}$, we have that each entry of $ \frac{\partial \tB^*}{ \partial (\gamma, s) }$ is smaller than or equal to $B_0$. Thus, the right hand side $\left\| \frac{\partial \tB^*}{ \partial (\gamma, s) } \tlam_B \right\| $ is smaller than or equal to $ (M+1 ) B_0 \times   \frac{ \sqrt{N}B_0}{\chi_0} =  \frac{ (M+1)\sqrt{N}B_0^2}{\chi_0}$. Thus, we have that $\left\| \tB^*  \frac{\partial \tlam_B}{ \partial (\gamma, s) } \right\| \le \frac{ (M+1)\sqrt{N}B_0^2}{\chi_0}$. 
	Again, by the minimum singular value of $\tB^*$, we have that 
	\begin{align}\label{equ2:lemma}
		\left \|   \frac{\partial \tlam_B}{ \partial (\gamma, s) } \right\| \le \frac{ (M+1)\sqrt{N}B_0^2}{\chi_0^2}.
	\end{align}
	
	Next, we consider the upper bound for $\left \|   \frac{\partial  x^*}{ \partial (\gamma, s) } \right\|$. Since $\tB^* x^* = \tilb_B$, by taking the derivative on both sides, we have that  $\frac{\partial \tB^*}{ \partial (\gamma, s) } x^* +  \tB^*  \frac{\partial x^*}{ \partial (\gamma, s) } = \frac{\partial \tilb_B}{ \partial (\gamma, s) }. $

	Since $\tilb_B$ is subvector of $\tilb$, we have that $\| \frac{\partial \tilb_B}{ \partial (\gamma, s) }\| \le \sqrt{N}$.
	By $\|x^*\|_2 \le \|x^*\|_1 \le K$, we have that  $\left\| \tB^*  \frac{\partial x^*}{ \partial (\gamma, s) } \right\| \le \sqrt{N} + (M+1)B_0 K$. Again, by the minimum singular value of $\tB^*$, we have that
	
	\begin{align}\label{equ3:lemma}
		\left \|   \frac{\partial x^*}{ \partial (\gamma, s) } \right\| \le \frac{\sqrt{N} + (M+1)B_0 K}{\chi_0}.
	\end{align}
	
	By taking the derivatives of $\frac{\partial \Psi}{\partial \gamma}$ and $\frac{\partial \Psi}{\partial s}$, we have that for the hessian matrix $\partial^2 \Psi(\gamma,s)$, each entry is no larger than $ \left \|   \frac{\partial \tlam_B}{ \partial (\gamma, s) } \right\|  (1 + B_0 K) + \|\tlam_B\| \cdot B_0 \cdot \left \|   \frac{\partial x^*}{ \partial (\gamma, s) } \right\|. $
	Combining the upper bounds in \eqref{equ2:lemma} and \eqref{equ3:lemma}, we have that each entry in the hessian matrix is no larger than $ \frac{ (M+1)\sqrt{N}B_0^2 (1 + B_0 K)}{\chi_0^2} +  \frac{ \sqrt{N}B_0^2}{\chi_0}   \frac{\sqrt{N} + (M+1)B_0 K}{\chi_0},$
	which is further no larger than $\frac{4 MKN B_0^3}{\chi_0^2}$. Thus, the operator norm of the hessian matrix $\partial^2 \Psi(\gamma,s)$ is no larger than $\frac{8 M^2 KN B_0^3}{\chi_0^2}$.
\end{proof}

{
\section{The Charnes Cooper transformation}\label{sec:cp}
Given a linear fractional programming
$$
\max_x\;\; \frac{c^\top x+\alpha}{d^\top x+\beta}\;\;\;\;s.t.\;\; Ax\leq b
$$
restricted to the domain $\{x: d^\top x+\beta>0\}$, the Charnes-Cooper transformation transforms it into an LP
$$
\max_{y,t}\;\; c^\top y+\alpha t\;\;\;\;s.t.\;\; Ay\leq bt,\;\; d^\top y+\beta t=1,\;\;t\geq 0
$$
where $y$ is a vector of the same length as $x$, and $t$ is a scalar. The transformed LP is equivalent to the original linear fractional programming, via the following equivalence:
$$
y=\frac{1}{d^\top x+\beta}x,\;\;\;\;\; t = \frac{1}{d^\top x+\beta}\;\;\;\;\;\text{and}\;\;\;\; x=\frac{1}{t} y.
$$
When the domain of $x$ is further restricted beyond $d^\top x+\beta>0$, the domain of $(y,t)$ is adjusted so that the above equivalence remains valid.
}

\section{Implementation and experimental details}\label{appsec:numerical}

Methods we implemented in Sec.~\ref{sec:numerical} are summarized as follows.
\begin{enumerate}
\item The first method, denoted as \textsf{Sampling-per-period}, obtains unbiased assortment samples $z_1,z_2,\cdots$ from $x^F$ that solves the fluid approximation $\Phi(\gamma)$ during every single time period. That is, the assortments offered to arriving customers are constantly changing from period to period;
\item The second method, denoted as \textsf{Sampling-per-epoch}, obtains unbiased assortment samples $z_1,z_2,\cdots$ from $x^F$ that solves the fluid approximation $\Phi(\gamma)$ whenever a no-purchase activity arises. That is, the algorithm offers the same assortment to arriving customers until the first customer purchased nothing, at which time the algorithm will obtain a fresh sample $z$ and updates its offered assortment for future customers.
\end{enumerate}
Note that both the above-mentioned baseline methods solve the fluid approximation $\Phi(\gamma)$ only once at the beginning of the algorithms,
and no re-solve or re-optimization is carried out during the assortment planning process.

We numerically synthesize problem instances with $N$ products, $M$ resources and capacity constraint $K$ as follows.
The utility parameters $\{v_i\}_{i=1}^N$ are uniformly distributed on $[0,1]$.
The normalized initial inventory levels $\{C_{0,j}\}_{j=1}^M/T$ are uniformly distributed on $[0,0.1]$.
For the resource consumption matrix $A\in\mathbb R^{N\times M}$, each element of $A$ is uniformly distributed on $[0,1/K]$.
Such synthesized problem instances achieve quite good balance between capacity and resource inventory constraints.

We also implemented an efficient BwK decomposition algorithm in the numerical results section.
For ease of presentation we assume that $x_1+\cdots+x_N=K$. The case of $x_1+\cdots+x_N<K$ can be easily handled by adding $K$ extra
``dummy'' variables, as shown in Algorithm \ref{alg:sampling-fluid}.
Without loss of generality we assume also that each $x_i$ belong to $(0,1)$, as those $x_i=1$ will always be selected and $x_i=0$ always excluded in the sampled assortments.
For a general $N\times N$ doubly stochastic matrix $\mat M$, the standard Birkhoff's algorithm uses up to $L=O(N^2)$ permutation matrices in its decomposition
and it takes $O(N^4)$ time to compute them. In this section we present a more efficient algorithm that decomposes $\mat M$ into at most $O(N)$ permutation matrices, with the total time complexity reduced to $O(N^2)$.
This algorithm is also what we implemented to conduct our numerical experiments in Sec.~\ref{sec:numerical}.

\begin{algorithm}[t]
\caption{An improved algorithm for computing BvN decomposition}
\label{alg:reduced-bvn}
\begin{algorithmic}[1]
\State \textbf{Input}: $x\in(0,1)^N$ such that $\sum_i x_i=K\in\mathbb N$;
\State \textbf{Output}: $\{\alpha_\ell,Z_\ell\}_{\ell=1}^L$ such that $\alpha_\ell\in[0,1]$, $\sum_\ell\alpha_\ell=1$, $Z_\ell\subseteq[N]$ and $|Z_\ell|=K$.
\State Initialize: $y_i=x_i$ and $\bar y_i=1-x_i$ for all $i\in[N]$;
\For{$\ell=1,2,\cdots,L$ until $y_i=\bar y_i=0$ for all $i\in[N]$}
	\State Let $S_\ell\subseteq[N]$, $|S_\ell|=K$ consist of $K$ elements with the largest $x_i$ values; let $\bar S_\ell=[N]\backslash S_\ell$;
	\State Let $\alpha_\ell=\min\{y_i|i\in S_\ell,y_i>0\}\wedge \min\{\bar y_i|i\in \bar S_\ell,\bar y_i>0\}$ and $z_{\ell i}=\vct 1\{i\in S_\ell\}$;
	\State Update $y_i\gets y_i-\alpha_\ell$ for each $i\in S_\ell$ and $\bar y_i\gets \bar y_i-\alpha_\ell$ for each $i\in \bar S_\ell$;
\EndFor
\end{algorithmic}
\end{algorithm}

Algorithm \ref{alg:reduced-bvn} gives a pseudocode description of the improved algorithm.
As we can see, each iteration of Algorithm \ref{alg:reduced-bvn} will reduce at least one $y_i$ or $\bar y_i$ to zero,
meaning that the algorithm will terminate in at most $N$ iterations.
Furthermore, the computational time for each iteration is linear in $N$. Hence, the total time complexity of Algorithm \ref{alg:reduced-bvn} is $O(N^2)$.

\end{appendices}

\ifdefined\unblind
%

\fi

\bibliographystyle{pomsref}
\bibliography{refs}
\end{document}